\documentclass{amsart}
\usepackage{longsalch, xy, eufrak}
\usepackage{breqn}
\xyoption{all}
\title{The cohomology of the height four Morava stabilizer group at large primes.}
\date{June 2016.}
\begin{document}
\maketitle
\begin{abstract}
This is an announcement of some new computational methods in stable homotopy theory, in particular, methods for using the cohomology of small-height Morava stabilizer groups to compute the cohomology of large-height Morava stabilizer groups. As an application, the cohomology of the height four Morava stabilizer group is computed at large primes (its rank turns out to be $3440$). Consequently we are able to formulate a plausible conjecture on the rank of the large-primary cohomology of the Morava stabilizer groups at all heights. 
\end{abstract}
\tableofcontents

\section{Introduction.}

This is an announcement of some new computations in homotopy theory, in particular, the computation of the cohomology of the height four Morava stabilizer group scheme at primes greater than five, and the computation of the homotopy groups of the $K(4)$-local Smith-Toda complex $V(3)$ at the primes at which it exists. 
These computations are possible using some new ``height-shifting'' techniques in computational homotopy theory, which decompose ``higher-height'' families of elements in stable homotopy in terms of copies of ``lower-height'' families of elements.
More specifically, in this document I construct of a sequence of spectral sequences whose first spectral sequence has as its input the cohomology of the height $n$ Morava stabilizer group scheme, and whose last spectral sequence has as its output the cohomology of the height $2n$ Morava stabilizer group scheme, at large primes. In the case $n=2$ this set of computations is actually accomplished at primes greater than five. I give quite a bit of detail about the computations\footnote{The spectral sequence computations described in this document are of a really substantial scale (see Proposition~\ref{coh of E 4 4 2}, for example), and although these spectral sequences have been worked through and checked once, it would not be unreasonable for the reader to regard the results of these computations as provisional until they have been completely written up and checked by others.} in this announcement, but some parts involve lengthy computer calculations (the cohomology algebra I compute in the end has $\mathbb{F}_p$-rank $3,440$) and due to the difficulty in presenting these kinds of calculations in a paper, these parts are not yet typed up; I plan to update electronically-available copies of this document with those details as they are typed up. (Once they are typed up in their entirety, this document will cease to be an ``announcement'' and will become an actual ``paper.'') 

The computation of the cohomology of Morava stabilizer groups is known to be of central importance in computational stable homotopy theory, but until now, very little has been done at heights greater than $2$, and as far as I know, nothing has been done at heights greater than $3$; largely this seems to be because computations at each given height start ``from scratch,'' i.e., there have been no available techniques to use ``lower-height'' computations as input for ``higher-height'' computations. Such techniques are developed and used in this document, and result in some patterns in the dimensions of the cohomologies of the Morava stabilizer groups at large heights, which suggest a conjecture about how the cohomology of the Morava stabilizer group scheme of a given height is ``built from'' the cohomology of the Morava stabilizer group scheme at smaller heights; see Conjecture~\ref{recursion conjecture}.

Now here is a bit of explanation about what Morava stabilizer group schemes are, and why they matter.
Let $p$ be a prime number and let $\mathbb{G}$ be a one-dimensional formal group law of $p$-height $n$ over $\mathbb{F}_p$.
The automorphism group scheme $\Aut(\mathbb{G})$ is called the {\em $p$-primary height $n$ Morava stabilizer group scheme.}
The group scheme cohomology $H^*(\Aut(\mathbb{G}); \mathbb{F}_p)$ is the input
for several spectral sequences of interest in algebraic topology:
\begin{itemize}
\item Beginning with $H^*(\Aut(\mathbb{G}); \mathbb{F}_p)$, a sequence of $n$ Bockstein spectral sequences computes
$\Cotor_{BP_*BP}^{*,*}(BP_*,v_n^{-1}BP_*/I_n^{\infty})$, the height $n$ layer in the $E_1$-term of the chromatic spectral sequence, which in turn converges
to $\Cotor_{BP_*BP}^{*,*}(BP_*,BP_*) \cong H_{fl}^{*,*}(\mathcal{M}_{fg}; \mathcal{O})$, the flat cohomology of the moduli stack of one-dimensional formal groups over $\Spec \mathbb{Z}_{(p)}$;
this cohomology is, in turn, the $E_2$-term of the Adams-Novikov spectral sequence, which converges to $\pi_*(S)_{(p)}$, the $p$-local stable homotopy groups of spheres.
See~\cite{MR860042} for this material (chapter 6 for the Bockstein spectral sequences, chapter 5 for the chromatic spectral sequence, chapter 4 for the Adams-Novikov spetral sequences).
\item Beginning with $H^*(\Aut(\mathbb{G}); \mathbb{F}_p)$, a sequence of $n$ Bockstein spectral sequences computes
$H^*(\Aut(\mathbb{G}); E(\mathbb{G})_*)$, the cohomology of $\Aut(\mathbb{G})$ with coefficients in the homotopy groups of the Morava $E$-theory spectrum $E(\mathbb{G})_*$; this is, in turn, the input for a descent spectral sequence converging to $\pi_*(L_{K(n)}S)$, the homotopy groups of the $K(n)$-local sphere spectrum. 
The long exact sequences induced in homotopy by ``fracture squares'' relating $L_{K(n)}S, L_{E(n)}S, L_{E(n-1)}S$, and $L_{E(n-1)}L_{K(n)}S$ allow one
(in principle) to inductively compute $L_{E(n)}S$ from knowledge of $L_{K(0)}S, L_{K(1)}S, \dots , L_{K(n)}S$, and the chromatic convergence theorem of Hopkins and Ravenel (see~\cite{MR1192553}) tells us that the $p$-local sphere spectrum is weakly equivalent to the homotopy limit $\holim_n L_{E(n)}S$.
See~\cite{MR1333942} and~\cite{MR2030586} for this material.
Since $E(\mathbb{G})_* \cong \Gamma(\Def(\mathbb{G}))[u^{\pm 1}]$, where $\Gamma(\Def(\mathbb{G}))$ is the ring of global section of the Lubin-Tate moduli space of deformations of $\mathbb{G}$, the computation of $H^*(\Aut(\mathbb{G}); E(\mathbb{G})_*)$ is perhaps of some independent interest in number theory.  
\item If the $p$-local Smith-Toda complex $V(n-1)$ exists, then $H^*(\Aut(\mathbb{G}); \mathbb{F}_p)[v_n^{-1}]$ is the input for the $E(n)$-based Adams spectral sequence which converges to $\pi_*(L_{E(n)}V(n-1)) \cong \pi_*(L_{K(n)}V(n-1))$. 
\item More generally, if $X$ is a $p$-local type $n$ finite complex (i.e., $X$ is $E(n-1)$-acyclic but not $E(n)$-acyclic), then $E(n)_*(X)$ admits a finite filtration whose filtration quotients are each isomorphic to $\mathbb{F}_p$, and this filtration gives rise to a spectral sequence
whose input is a number of copies of $H^*(\Aut(\mathbb{G}); \mathbb{F}_p)$ and whose output is the $E_2$-term $\Cotor^{*,*}_{E(n)_*E(n)}(E(n)_*,E(n)_*(X))$ of the $E(n)$-based Adams spectral sequence which converges to $\pi_*(L_{E(n)} X)$.
\item There are further variations and combinations of these methods for computing $K(n)$-local, $E(n)$-local, or global stable homotopy groups of various finite $CW$-complexes.
\end{itemize}
So, in order to make computations of stable homotopy groups, it is very, very helpful to be able to compute $H^*(\Aut(\mathbb{G}); \mathbb{F}_p)$. 
This has been accomplished when $n=1$ (classical; see Theorem~6.3.21 of~\cite{MR860042}, but one can even extract it from~\cite{MR1697859} without any mention of homotopy theory at all), when $n=2$ (see Theorems~6.3.22 and~6.3.24 of~\cite{MR860042}, and Theorem~6.3.27 for the $p=2$ and $n=2$ answer modulo multiplicative extensions), and when $n=3$ and $p>3$ (see Theorem~6.3.34 of~\cite{MR860042}).

In this document, we construct some spectral sequences which are designed to be used in conjunction with Ravenel's spectral sequences, from section 6.3 of~\cite{MR860042}, to compute the cohomology of large-height Morava stabilizer group schemes in terms of lower-height Morava stabilizer group schemes. Here are the main steps from our Strategy~\ref{main strategy}:
\begin{description}
\item[From a height $n$ formal group to a height $n$ formal module] Let $n$ be a positive integer and let $p$ be a sufficiently large prime (for example, $p>5$ suffices when $n=2$, and for arbitrary $n$, $p>2n+1$ ``usually'' suffices; the bound on $p$ is so that the group scheme cohomology $H^*(\strictAut(\mathbb{G}_{1/n}) ; \mathbb{F}_p)$ of the height $n$ Morava stabilizer group scheme $\strictAut(\mathbb{G}_{1/n})$ is isomorphic to the cohomology of a certain DGA, as in section~6.3 of~\cite{MR860042}). Then there exists a spectral sequence whose input is a tensor product of $H^*(\strictAut(\mathbb{G}_{1/n}) ; \mathbb{F}_p)$ with an exterior algebra on $n^2$ generators, and whose
output is $H^*(\strictAut(\mathbb{G}^{\hat{\mathbb{Z}}_p\left[ \sqrt{p}\right]}_{1/n}) ; \mathbb{F}_p)$, the group scheme cohomology  of the strict automorphism group scheme of a $\hat{\mathbb{Z}}_p\left[ \sqrt{p}\right]$-height $n$ formal $\hat{\mathbb{Z}}_p\left[ \sqrt{p}\right]$-module over $\mathbb{F}_p$. 

(Conceptually, this is appealing: $H^*(\strictAut(\mathbb{G}^{\hat{\mathbb{Z}}_p\left[ \sqrt{p}\right]}_{1/n}) ; \mathbb{F}_p)$ is ``built from'' $2^{(n^2)}$ copies of $H^*(\strictAut(\mathbb{G}_{1/n}) ; \mathbb{F}_p)$. But for practical computations, this spectral sequence can be broken down into a sequence of $n$ intermediate spectral sequences which are much smaller and more tractable; the $n=2$ case of these intermediate spectral sequence calculations are Propositions~\ref{coh of E 4 3 0} and  \ref{coh of E 4 4 0} in the present paper.)
\item[From a height $n$ formal module to a height $2n$ formal group] We then use a ``height-shifting spectral sequence'' whose input is
the tensor product of $H^*(\strictAut(\mathbb{G}^{\hat{\mathbb{Z}}_p\left[ \sqrt{p}\right]}_{1/n}) ; \mathbb{F}_p)$ with an exterior algebra on $2n^2$ generators, and whose output is the cohomology of a certain differential graded algebra $\mathcal{E}(2n,2n,2n)$, which, if $p>2n+1$, is isomorphic to the cohomology of the restricted Lie algebra $PE^0\mathbb{F}_p[\strictAut(\mathbb{G}_{1/2n})]$ of primitives in the associated graded Hopf algebra of Ravenel's filtration on the profinite group ring $\mathbb{F}_p[\strictAut(\mathbb{G}_{1/2n})]$.
The cohomology of this restricted Lie algebra is, in turn, the input for Ravenel's May spectral sequence (see Theorem~\ref{ravenels may spectral sequences}, or see Theorem~6.3.4 of~\cite{MR860042}) which converges to the group scheme cohomology $H^*(\strictAut(\mathbb{G}_{1/2n}); \mathbb{F}_p)$ of the height $2n$ Morava stabilizer group scheme.

(Again, conceptually, this is appealing: $H^*(\strictAut(\mathbb{G}_{1/2n}) ; \mathbb{F}_p)$ is ``built from'' $2^{2(n^2)}$ copies of $H^*(\strictAut(\mathbb{G}^{\hat{\mathbb{Z}}_p\left[ \sqrt{p}\right]}_{1/n}) ; \mathbb{F}_p)$. But for practical computations, this spectral sequence can be broken down into a sequence of $2n$ intermediate spectral sequences which are, again, much smaller and more tractable; $n=2$ cases of these intermediate spectral sequence calculations are Propositions~\ref{coh of E 4 4 1} and  \ref{coh of E 4 4 2} in the present paper.)
\end{description}
A ``formal $A$-module'' is a formal group with complex multiplication by $A$; the definition is recalled in Definition~\ref{def of formal module}, and a few basic properties of formal modules are in the definitions and theorems immediately following Definition~\ref{def of formal module}.

Carrying out these computations in the case $n=2$ (i.e., to compute the cohomology of the height $4$ Morava stabilizer group scheme using, as input, the cohomology of the height $2$ Morava stabilizer group scheme, at primes $p>5$), we get:
\begin{itemize}
\item In Proposition~\ref{cohomology of ht 2 morava stab grp} we compute the cohomology of the height $2$ Morava stabilizer group scheme (which was already known; see e.g. Theorem~6.3.22 of~\cite{MR860042}); $H^*(\strictAut(\mathbb{G}_{1/2}); \mathbb{F}_p)$ has cohomological dimension $4$, total rank (i.e., dimension as an $\mathbb{F}_p$-vector space) $12$, and Poincar\'{e} series 
\[ (1+s)^2(1+s+s^2) = 1+3s+4s^2+3s^3+s^4.\]

For topological applications, one also needs to know topological degrees of cohomology classes. These topological degrees are are all divisible by $2(p-1)$, and due to $v_2$-periodicity, are only defined modulo $2(p^2-1)$. If we write a two-variable Poincar\'{e} series for $H^*(\strictAut(\mathbb{G}_{1/2}); \mathbb{F}_p)$, using $s$ to indicate cohomological dimension and using $t$ to indicate topological dimension divided by $2(p-1)$ (e.g. $2s^2t^{p^3}$ is a $2$-dimensional $\mathbb{F}_p$-vector space in cohomological degree $2$ and topological degree $2p^3(p-1)$), then the resulting Poincar\'{e} series is
\[ (1 + s)(1 + st + st^p + s^2t + s^2t^p + s^3).\]
\item In Proposition~\ref{coh of ht 2 fm} we compute the  cohomology of the strict automorphism group scheme of a $\hat{\mathbb{Z}}_p\left[\sqrt{p}\right]$-height $2$ formal $\hat{\mathbb{Z}}_p\left[\sqrt{p}\right]$-module. This is a new computation. We get that $H^*(\strictAut(\mathbb{G}^{\hat{\mathbb{Z}}_p\left[\sqrt{p}\right]}_{1/2}); \mathbb{F}_p)$ has cohomological dimension $8$, total rank $80$, and Poincar\'{e} series 
\begin{dmath*} (1+s)^4(1+3s^2+s^4) =  1+ 4s + 9s^2 + 16s^3 + 20s^4 + 16s^5 + 9s^6 + 4s^7 + s^8.\end{dmath*}

If we write a two-variable Poincar\'{e} series for $H^*(\strictAut(\mathbb{G}^{\hat{\mathbb{Z}}_p\left[\sqrt{p}\right]}_{1/2}); \mathbb{F}_p)$, then the resulting Poincar\'{e} series is
\begin{dmath*} (1+st^{p+1})(1+s)(1 + st + st^p + s^2t^{2} + s^2t^{2p} + s^2t + s^2t^p + 2s^3 + 2s^3t^2 + 2s^3t^{2p} + s^4t + s^4t^p + s^4t^2 + s^4t^{2p} + s^5t + s^5t^p + s^6).\end{dmath*}
\item We carry out (in Propositions~\ref{coh of E 4 4 1} and~\ref{coh of E 4 4 2}) two of the remaining five spectral sequences required to compute the cohomology of the height four Morava stabilizer group scheme. As is evident from the proofs presented in this document, these computations are not too difficult to do by hand.
The remaining spectral sequences were also run primarily by hand but for some parts computer assistance was used (specifically, I used a computer to find cocycle representatives for certain cohomology classes, in order to compute spectral sequence differentials); I plan to update the electronically-available copies of this document with full details of those computations as I get them typed up.
Here are the results I arrived at from running the remaining spectral sequences:
the Poincar\'{e} series of $H^*(\strictAut(\mathbb{G}_{1/4}); \mathbb{F}_p)$, the mod $p$ group scheme cohomology of the strict automorphism group scheme of a height $4$ formal group over $\mathbb{F}_p$, using $s$ to denote cohomological degree, is:
\begin{dmath*}  
(1+s)^4(s^4-s^3+5s^2-x+1)(1+2s+5s^2+9s^3+9s^4+9s^5+5s^6+2s^7+s^8) = 
1+  5 s^{ 1 }  +  18 s^{ 2 }  +  55 s^{ 3 }  +  129 s^{ 4 }  +  249 s^{ 5 }  +  409 s^{ 6 }  +  551 s^{ 7 } + 606 s^8 + 551 s^9 + 409 s^{10} + 249 s^{11} + 129 s^{12} + 55 s^{13} + 18 s^{14} + 5 s^{15} + s^{16} .
\end{dmath*}

The two-variable Poincar\'{e} series of $H^*(\strictAut(\mathbb{G}_{1/4}); \mathbb{F}_p)$, using $s$ to denote cohomological degree and $t$ to denote topological degree divided by $2(p-1)$, 
is:
\begin{dmath*}
1 + s^1 \left( 1  + t^{ p^3 }  + t^{ p^2 }  + t^{ p }  + t \right)
 + s^2 \left(  2 t^{ p^3 }  +  2 t^{ p^2 }  + t^{ p^2 + 2 p^3 }  + t^{ 2 p^2 + p^3 }  +  2 t^{ p }  + t^{ p + p^3 }  + t^{ p + 2 p^2 }  + t^{ 2 p + p^2 }  +  2 t  + t^{ 1 + 2 p^3 }  + t^{ 1 + p^2 }  + t^{ 1 + 2 p }  + t^{ 2 + p^3 }  + t^{ 2 + p } \right)
 + s^3 \left( 1  + t^{ p^3 }  + t^{ 2 p^3 }  + t^{ p^2 }  + t^{ p^2 + p^3 }  +  2 t^{ p^2 + 2 p^3 }  + t^{ 2 p^2 }  +  2 t^{ 2 p^2 + p^3 }  + t^{ 2 p^2 + 2 p^3 }  + t^{ p }  +  3 t^{ p + p^3 }  + t^{ p + p^2 }  +  2 t^{ p + 2 p^2 }  + t^{ p + 2 p^2 + 3 p^3 }  + t^{ p + 3 p^2 + p^3 }  + t^{ 2 p }  +  2 t^{ 2 p + p^2 }  + t^{ 2 p + p^2 + 2 p^3 }  + t^{ 2 p + 2 p^2 }  + t^{ 3 p + 2 p^2 + p^3 }  + t  + t^{ 1 + p^3 }  +  2 t^{ 1 + 2 p^3 }  +  3 t^{ 1 + p^2 }  + t^{ 1 + p^2 + 3 p^3 }  + t^{ 1 + 3 p^2 + 2 p^3 }  + t^{ 1 + p }  +  2 t^{ 1 + 2 p }  + t^{ 1 + 2 p + 2 p^3 }  + t^{ 1 + 2 p + 3 p^2 }  + t^{ 1 + 3 p + p^2 }  + t^{ 2 }  +  2 t^{ 2 + p^3 }  + t^{ 2 + 2 p^3 }  + t^{ 2 + 2 p^2 + p^3 }  +  2 t^{ 2 + p }  + t^{ 2 + p + 3 p^3 }  + t^{ 2 + p + 2 p^2 }  + t^{ 2 + 2 p }  + t^{ 2 + 3 p + p^3 }  + t^{ 3 + p^2 + 2 p^3 }  + t^{ 3 + p + p^3 }  + t^{ 3 + 2 p + p^2 } \right)
 + s^4 \left( 1  + t^{ p^3 }  +  2 t^{ 2 p^3 }  + t^{ p^2 }  +  2 t^{ p^2 + p^3 }  + t^{ p^2 + 2 p^3 }  + t^{ p^2 + 3 p^3 }  +  2 t^{ 2 p^2 }  + t^{ 2 p^2 + p^3 }  +  3 t^{ 2 p^2 + 2 p^3 }  + t^{ 3 p^2 + p^3 }  + t^{ p }  +  3 t^{ p + p^3 }  + t^{ p + 2 p^3 }  +  2 t^{ p + p^2 }  + t^{ p + p^2 + 2 p^3 }  + t^{ p + 2 p^2 }  + t^{ p + 2 p^2 + p^3 }  +  2 t^{ p + 2 p^2 + 3 p^3 }  + t^{ p + 3 p^2 }  +  3 t^{ p + 3 p^2 + p^3 }  + t^{ p + 3 p^2 + 3 p^3 }  +  2 t^{ 2 p }  + t^{ 2 p + p^3 }  + t^{ 2 p + 2 p^3 }  + t^{ 2 p + p^2 }  + t^{ 2 p + p^2 + p^3 }  +  2 t^{ 2 p + p^2 + 2 p^3 }  +  3 t^{ 2 p + 2 p^2 }  + t^{ 2 p + 2 p^2 + 3 p^3 }  + t^{ 2 p + 4 p^2 + 2 p^3 }  + t^{ 3 p + p^2 }  +  2 t^{ 3 p + 2 p^2 + p^3 }  + t^{ 3 p + 2 p^2 + 2 p^3 }  + t^{ 3 p + 3 p^2 + p^3 }  + t  +  2 t^{ 1 + p^3 }  + t^{ 1 + 2 p^3 }  + t^{ 1 + 3 p^3 }  +  3 t^{ 1 + p^2 }  + t^{ 1 + p^2 + 2 p^3 }  +  3 t^{ 1 + p^2 + 3 p^3 }  + t^{ 1 + 2 p^2 }  + t^{ 1 + 2 p^2 + p^3 }  +  2 t^{ 1 + 3 p^2 + 2 p^3 }  + t^{ 1 + 3 p^2 + 3 p^3 }  +  2 t^{ 1 + p }  + t^{ 1 + p + 2 p^3 }  + t^{ 1 + p + 2 p^2 }  + t^{ 1 + 2 p }  + t^{ 1 + 2 p + p^3 }  +  2 t^{ 1 + 2 p + 2 p^3 }  + t^{ 1 + 2 p + p^2 }  +  2 t^{ 1 + 2 p + 3 p^2 }  + t^{ 1 + 3 p }  +  3 t^{ 1 + 3 p + p^2 }  + t^{ 1 + 3 p + 3 p^2 }  +  2 t^{ 2 }  + t^{ 2 + p^3 }  +  3 t^{ 2 + 2 p^3 }  + t^{ 2 + p^2 }  + t^{ 2 + p^2 + p^3 }  + t^{ 2 + 2 p^2 }  +  2 t^{ 2 + 2 p^2 + p^3 }  + t^{ 2 + 2 p^2 + 4 p^3 }  + t^{ 2 + 3 p^2 + 2 p^3 }  + t^{ 2 + p }  + t^{ 2 + p + p^3 }  +  2 t^{ 2 + p + 3 p^3 }  + t^{ 2 + p + p^2 }  +  2 t^{ 2 + p + 2 p^2 }  +  3 t^{ 2 + 2 p }  + t^{ 2 + 2 p + 3 p^3 }  + t^{ 2 + 2 p + 3 p^2 }  +  2 t^{ 2 + 3 p + p^3 }  + t^{ 2 + 3 p + 2 p^3 }  + t^{ 2 + 4 p + 2 p^2 }  + t^{ 3 + p^3 }  +  2 t^{ 3 + p^2 + 2 p^3 }  + t^{ 3 + p^2 + 3 p^3 }  + t^{ 3 + 2 p^2 + 2 p^3 }  + t^{ 3 + p }  +  3 t^{ 3 + p + p^3 }  + t^{ 3 + p + 3 p^3 }  +  2 t^{ 3 + 2 p + p^2 }  + t^{ 3 + 2 p + 2 p^2 }  + t^{ 3 + 3 p + p^3 }  + t^{ 3 + 3 p + p^2 }  + t^{ 4 + 2 p + 2 p^3 } \right)
 + s^5 \left( 1  +  2 t^{ p^3 }  + t^{ 2 p^3 }  + t^{ 3 p^3 }  +  2 t^{ p^2 }  + t^{ p^2 + p^3 }  +  2 t^{ p^2 + 2 p^3 }  +  3 t^{ p^2 + 3 p^3 }  + t^{ 2 p^2 }  +  2 t^{ 2 p^2 + p^3 }  +  3 t^{ 2 p^2 + 2 p^3 }  + t^{ 3 p^2 }  +  3 t^{ 3 p^2 + p^3 }  + t^{ 3 p^2 + 3 p^3 }  +  2 t^{ p }  +  2 t^{ p + p^3 }  +  4 t^{ p + 2 p^3 }  + t^{ p + p^2 }  + t^{ p + p^2 + p^3 }  +  2 t^{ p + p^2 + 2 p^3 }  +  2 t^{ p + 2 p^2 }  +  2 t^{ p + 2 p^2 + p^3 }  + t^{ p + 2 p^2 + 3 p^3 }  +  3 t^{ p + 3 p^2 }  +  3 t^{ p + 3 p^2 + p^3 }  + t^{ p + 3 p^2 + 2 p^3 }  +  3 t^{ p + 3 p^2 + 3 p^3 }  + t^{ p + 4 p^2 + 2 p^3 }  + t^{ 2 p }  +  4 t^{ 2 p + p^3 }  +  2 t^{ 2 p + 2 p^3 }  +  2 t^{ 2 p + p^2 }  +  2 t^{ 2 p + p^2 + p^3 }  + t^{ 2 p + p^2 + 2 p^3 }  + t^{ 2 p + p^2 + 3 p^3 }  +  3 t^{ 2 p + 2 p^2 }  + t^{ 2 p + 2 p^2 + 2 p^3 }  +  3 t^{ 2 p + 2 p^2 + 3 p^3 }  + t^{ 2 p + 3 p^2 + p^3 }  + t^{ 2 p + 4 p^2 + p^3 }  +  2 t^{ 2 p + 4 p^2 + 2 p^3 }  + t^{ 2 p + 4 p^2 + 3 p^3 }  + t^{ 3 p }  +  3 t^{ 3 p + p^2 }  + t^{ 3 p + p^2 + 2 p^3 }  + t^{ 3 p + 2 p^2 + p^3 }  +  3 t^{ 3 p + 2 p^2 + 2 p^3 }  + t^{ 3 p + 3 p^2 }  +  3 t^{ 3 p + 3 p^2 + p^3 }  + t^{ 3 p + 3 p^2 + 3 p^3 }  + t^{ 3 p + 4 p^2 + 2 p^3 }  +  2 t  + t^{ 1 + p^3 }  +  2 t^{ 1 + 2 p^3 }  +  3 t^{ 1 + 3 p^3 }  +  2 t^{ 1 + p^2 }  + t^{ 1 + p^2 + p^3 }  +  2 t^{ 1 + p^2 + 2 p^3 }  +  3 t^{ 1 + p^2 + 3 p^3 }  +  4 t^{ 1 + 2 p^2 }  +  2 t^{ 1 + 2 p^2 + p^3 }  + t^{ 1 + 2 p^2 + 3 p^3 }  + t^{ 1 + 2 p^2 + 4 p^3 }  + t^{ 1 + 3 p^2 + 2 p^3 }  +  3 t^{ 1 + 3 p^2 + 3 p^3 }  + t^{ 1 + p }  + t^{ 1 + p + p^3 }  +  2 t^{ 1 + p + 2 p^3 }  + t^{ 1 + p + p^2 }  +  2 t^{ 1 + p + 2 p^2 }  +  2 t^{ 1 + 2 p }  +  2 t^{ 1 + 2 p + p^3 }  + t^{ 1 + 2 p + 2 p^3 }  + t^{ 1 + 2 p + 3 p^3 }  +  2 t^{ 1 + 2 p + p^2 }  + t^{ 1 + 2 p + 3 p^2 }  +  3 t^{ 1 + 3 p }  + t^{ 1 + 3 p + 2 p^3 }  +  3 t^{ 1 + 3 p + p^2 }  + t^{ 1 + 3 p + 2 p^2 }  +  3 t^{ 1 + 3 p + 3 p^2 }  + t^{ 1 + 4 p + 2 p^2 }  + t^{ 2 }  +  2 t^{ 2 + p^3 }  +  3 t^{ 2 + 2 p^3 }  +  4 t^{ 2 + p^2 }  +  2 t^{ 2 + p^2 + p^3 }  + t^{ 2 + p^2 + 3 p^3 }  + t^{ 2 + p^2 + 4 p^3 }  +  2 t^{ 2 + 2 p^2 }  + t^{ 2 + 2 p^2 + p^3 }  + t^{ 2 + 2 p^2 + 2 p^3 }  +  2 t^{ 2 + 2 p^2 + 4 p^3 }  + t^{ 2 + 3 p^2 + p^3 }  +  3 t^{ 2 + 3 p^2 + 2 p^3 }  + t^{ 2 + 3 p^2 + 4 p^3 }  +  2 t^{ 2 + p }  +  2 t^{ 2 + p + p^3 }  + t^{ 2 + p + 3 p^3 }  +  2 t^{ 2 + p + p^2 }  + t^{ 2 + p + 2 p^2 }  + t^{ 2 + p + 3 p^2 }  +  3 t^{ 2 + 2 p }  + t^{ 2 + 2 p + 2 p^3 }  +  3 t^{ 2 + 2 p + 3 p^3 }  + t^{ 2 + 2 p + 2 p^2 }  +  3 t^{ 2 + 2 p + 3 p^2 }  + t^{ 2 + 3 p + p^3 }  +  3 t^{ 2 + 3 p + 2 p^3 }  + t^{ 2 + 3 p + p^2 }  + t^{ 2 + 4 p + p^2 }  +  2 t^{ 2 + 4 p + 2 p^2 }  + t^{ 2 + 4 p + 3 p^2 }  + t^{ 3 }  +  3 t^{ 3 + p^3 }  + t^{ 3 + 3 p^3 }  + t^{ 3 + p^2 + 2 p^3 }  +  3 t^{ 3 + p^2 + 3 p^3 }  + t^{ 3 + 2 p^2 + p^3 }  +  3 t^{ 3 + 2 p^2 + 2 p^3 }  + t^{ 3 + 2 p^2 + 4 p^3 }  + t^{ 3 + 3 p^2 + 3 p^3 }  +  3 t^{ 3 + p }  +  3 t^{ 3 + p + p^3 }  + t^{ 3 + p + 2 p^3 }  +  3 t^{ 3 + p + 3 p^3 }  + t^{ 3 + p + 2 p^2 }  + t^{ 3 + 2 p + p^3 }  + t^{ 3 + 2 p + p^2 }  +  3 t^{ 3 + 2 p + 2 p^2 }  + t^{ 3 + 3 p }  +  3 t^{ 3 + 3 p + p^3 }  + t^{ 3 + 3 p + 3 p^3 }  +  3 t^{ 3 + 3 p + p^2 }  + t^{ 3 + 3 p + 3 p^2 }  + t^{ 3 + 4 p + 2 p^2 }  + t^{ 4 + p + 2 p^3 }  + t^{ 4 + 2 p + p^3 }  +  2 t^{ 4 + 2 p + 2 p^3 }  + t^{ 4 + 2 p + 3 p^3 }  + t^{ 4 + 3 p + 2 p^3 } \right) \end{dmath*}\begin{dmath*} 
 + s^6 \left( 1  +  2 t^{ p^3 }  +  2 t^{ 2 p^3 }  +  3 t^{ 3 p^3 }  +  2 t^{ p^2 }  + t^{ p^2 + p^3 }  +  4 t^{ p^2 + 2 p^3 }  +  3 t^{ p^2 + 3 p^3 }  +  2 t^{ 2 p^2 }  +  4 t^{ 2 p^2 + p^3 }  +  2 t^{ 2 p^2 + 2 p^3 }  + t^{ 2 p^2 + 3 p^3 }  +  3 t^{ 3 p^2 }  +  3 t^{ 3 p^2 + p^3 }  + t^{ 3 p^2 + 2 p^3 }  +  3 t^{ 3 p^2 + 3 p^3 }  +  2 t^{ p }  +  3 t^{ p + p^3 }  +  5 t^{ p + 2 p^3 }  + t^{ p + 3 p^3 }  + t^{ p + p^2 }  +  2 t^{ p + p^2 + p^3 }  + t^{ p + p^2 + 2 p^3 }  + t^{ p + p^2 + 3 p^3 }  +  4 t^{ p + 2 p^2 }  + t^{ p + 2 p^2 + p^3 }  + t^{ p + 2 p^2 + 2 p^3 }  + t^{ p + 2 p^2 + 3 p^3 }  +  3 t^{ p + 3 p^2 }  +  2 t^{ p + 3 p^2 + p^3 }  +  3 t^{ p + 3 p^2 + 2 p^3 }  +  3 t^{ p + 3 p^2 + 3 p^3 }  + t^{ p + 4 p^2 + p^3 }  +  3 t^{ p + 4 p^2 + 2 p^3 }  + t^{ p + 4 p^2 + 3 p^3 }  +  2 t^{ 2 p }  +  5 t^{ 2 p + p^3 }  + t^{ 2 p + 2 p^3 }  +  4 t^{ 2 p + p^2 }  + t^{ 2 p + p^2 + p^3 }  +  3 t^{ 2 p + p^2 + 2 p^3 }  +  3 t^{ 2 p + p^2 + 3 p^3 }  +  2 t^{ 2 p + 2 p^2 }  + t^{ 2 p + 2 p^2 + p^3 }  +  2 t^{ 2 p + 2 p^2 + 2 p^3 }  +  3 t^{ 2 p + 2 p^2 + 3 p^3 }  + t^{ 2 p + 3 p^2 }  +  3 t^{ 2 p + 3 p^2 + p^3 }  + t^{ 2 p + 3 p^2 + 3 p^3 }  +  3 t^{ 2 p + 4 p^2 + p^3 }  + t^{ 2 p + 4 p^2 + 2 p^3 }  +  3 t^{ 2 p + 4 p^2 + 3 p^3 }  +  3 t^{ 3 p }  + t^{ 3 p + p^3 }  +  3 t^{ 3 p + p^2 }  + t^{ 3 p + p^2 + p^3 }  +  3 t^{ 3 p + p^2 + 2 p^3 }  + t^{ 3 p + 2 p^2 }  + t^{ 3 p + 2 p^2 + p^3 }  +  3 t^{ 3 p + 2 p^2 + 2 p^3 }  + t^{ 3 p + 2 p^2 + 3 p^3 }  +  3 t^{ 3 p + 3 p^2 }  +  3 t^{ 3 p + 3 p^2 + p^3 }  + t^{ 3 p + 3 p^2 + 2 p^3 }  +  3 t^{ 3 p + 3 p^2 + 3 p^3 }  + t^{ 3 p + 4 p^2 + p^3 }  +  3 t^{ 3 p + 4 p^2 + 2 p^3 }  + t^{ 3 p + 4 p^2 + 3 p^3 }  +  2 t  + t^{ 1 + p^3 }  +  4 t^{ 1 + 2 p^3 }  +  3 t^{ 1 + 3 p^3 }  +  3 t^{ 1 + p^2 }  +  2 t^{ 1 + p^2 + p^3 }  + t^{ 1 + p^2 + 2 p^3 }  +  2 t^{ 1 + p^2 + 3 p^3 }  + t^{ 1 + p^2 + 4 p^3 }  +  5 t^{ 1 + 2 p^2 }  + t^{ 1 + 2 p^2 + p^3 }  + t^{ 1 + 2 p^2 + 2 p^3 }  +  3 t^{ 1 + 2 p^2 + 3 p^3 }  +  3 t^{ 1 + 2 p^2 + 4 p^3 }  + t^{ 1 + 3 p^2 }  + t^{ 1 + 3 p^2 + p^3 }  + t^{ 1 + 3 p^2 + 2 p^3 }  +  3 t^{ 1 + 3 p^2 + 3 p^3 }  + t^{ 1 + 3 p^2 + 4 p^3 }  + t^{ 1 + p }  +  2 t^{ 1 + p + p^3 }  + t^{ 1 + p + 2 p^3 }  + t^{ 1 + p + 3 p^3 }  +  2 t^{ 1 + p + p^2 }  + t^{ 1 + p + 2 p^2 }  + t^{ 1 + p + 3 p^2 }  +  4 t^{ 1 + 2 p }  + t^{ 1 + 2 p + p^3 }  +  3 t^{ 1 + 2 p + 2 p^3 }  +  3 t^{ 1 + 2 p + 3 p^3 }  + t^{ 1 + 2 p + p^2 }  + t^{ 1 + 2 p + 2 p^2 }  + t^{ 1 + 2 p + 3 p^2 }  +  3 t^{ 1 + 3 p }  + t^{ 1 + 3 p + p^3 }  +  3 t^{ 1 + 3 p + 2 p^3 }  +  2 t^{ 1 + 3 p + p^2 }  +  3 t^{ 1 + 3 p + 2 p^2 }  +  3 t^{ 1 + 3 p + 3 p^2 }  + t^{ 1 + 4 p + p^2 }  +  3 t^{ 1 + 4 p + 2 p^2 }  + t^{ 1 + 4 p + 3 p^2 }  +  2 t^{ 2 }  +  4 t^{ 2 + p^3 }  +  2 t^{ 2 + 2 p^3 }  + t^{ 2 + 3 p^3 }  +  5 t^{ 2 + p^2 }  + t^{ 2 + p^2 + p^3 }  + t^{ 2 + p^2 + 2 p^3 }  +  3 t^{ 2 + p^2 + 3 p^3 }  +  3 t^{ 2 + p^2 + 4 p^3 }  + t^{ 2 + 2 p^2 }  +  3 t^{ 2 + 2 p^2 + p^3 }  +  2 t^{ 2 + 2 p^2 + 2 p^3 }  + t^{ 2 + 2 p^2 + 4 p^3 }  +  3 t^{ 2 + 3 p^2 + p^3 }  +  3 t^{ 2 + 3 p^2 + 2 p^3 }  + t^{ 2 + 3 p^2 + 3 p^3 }  +  3 t^{ 2 + 3 p^2 + 4 p^3 }  +  4 t^{ 2 + p }  + t^{ 2 + p + p^3 }  + t^{ 2 + p + 2 p^3 }  + t^{ 2 + p + 3 p^3 }  + t^{ 2 + p + p^2 }  +  3 t^{ 2 + p + 2 p^2 }  +  3 t^{ 2 + p + 3 p^2 }  +  2 t^{ 2 + 2 p }  + t^{ 2 + 2 p + p^3 }  +  2 t^{ 2 + 2 p + 2 p^3 }  +  3 t^{ 2 + 2 p + 3 p^3 }  + t^{ 2 + 2 p + p^2 }  +  2 t^{ 2 + 2 p + 2 p^2 }  +  3 t^{ 2 + 2 p + 3 p^2 }  + t^{ 2 + 3 p }  + t^{ 2 + 3 p + p^3 }  +  3 t^{ 2 + 3 p + 2 p^3 }  + t^{ 2 + 3 p + 3 p^3 }  +  3 t^{ 2 + 3 p + p^2 }  + t^{ 2 + 3 p + 3 p^2 }  +  3 t^{ 2 + 4 p + p^2 }  + t^{ 2 + 4 p + 2 p^2 }  +  3 t^{ 2 + 4 p + 3 p^2 }  +  3 t^{ 3 }  +  3 t^{ 3 + p^3 }  + t^{ 3 + 2 p^3 }  +  3 t^{ 3 + 3 p^3 }  + t^{ 3 + p^2 }  + t^{ 3 + p^2 + p^3 }  + t^{ 3 + p^2 + 2 p^3 }  +  3 t^{ 3 + p^2 + 3 p^3 }  + t^{ 3 + p^2 + 4 p^3 }  +  3 t^{ 3 + 2 p^2 + p^3 }  +  3 t^{ 3 + 2 p^2 + 2 p^3 }  + t^{ 3 + 2 p^2 + 3 p^3 }  +  3 t^{ 3 + 2 p^2 + 4 p^3 }  + t^{ 3 + 3 p^2 + 2 p^3 }  +  3 t^{ 3 + 3 p^2 + 3 p^3 }  + t^{ 3 + 3 p^2 + 4 p^3 }  +  3 t^{ 3 + p }  +  2 t^{ 3 + p + p^3 }  +  3 t^{ 3 + p + 2 p^3 }  +  3 t^{ 3 + p + 3 p^3 }  + t^{ 3 + p + p^2 }  +  3 t^{ 3 + p + 2 p^2 }  + t^{ 3 + 2 p }  +  3 t^{ 3 + 2 p + p^3 }  + t^{ 3 + 2 p + 3 p^3 }  + t^{ 3 + 2 p + p^2 }  +  3 t^{ 3 + 2 p + 2 p^2 }  + t^{ 3 + 2 p + 3 p^2 }  +  3 t^{ 3 + 3 p }  +  3 t^{ 3 + 3 p + p^3 }  + t^{ 3 + 3 p + 2 p^3 }  +  3 t^{ 3 + 3 p + 3 p^3 }  +  3 t^{ 3 + 3 p + p^2 }  + t^{ 3 + 3 p + 2 p^2 }  +  3 t^{ 3 + 3 p + 3 p^2 }  + t^{ 3 + 4 p + p^2 }  +  3 t^{ 3 + 4 p + 2 p^2 }  + t^{ 3 + 4 p + 3 p^2 }  + t^{ 4 + p + p^3 }  +  3 t^{ 4 + p + 2 p^3 }  + t^{ 4 + p + 3 p^3 }  +  3 t^{ 4 + 2 p + p^3 }  + t^{ 4 + 2 p + 2 p^3 }  +  3 t^{ 4 + 2 p + 3 p^3 }  + t^{ 4 + 3 p + p^3 }  +  3 t^{ 4 + 3 p + 2 p^3 }  + t^{ 4 + 3 p + 3 p^3 } \right) \end{dmath*}\begin{dmath*} 
 + s^7 \left( 1  +  2 t^{ p^3 }  +  4 t^{ 2 p^3 }  +  3 t^{ 3 p^3 }  +  2 t^{ p^2 }  +  2 t^{ p^2 + p^3 }  +  2 t^{ p^2 + 2 p^3 }  +  2 t^{ p^2 + 3 p^3 }  +  4 t^{ 2 p^2 }  +  2 t^{ 2 p^2 + p^3 }  +  4 t^{ 2 p^2 + 2 p^3 }  +  3 t^{ 2 p^2 + 3 p^3 }  +  3 t^{ 3 p^2 }  +  2 t^{ 3 p^2 + p^3 }  +  3 t^{ 3 p^2 + 2 p^3 }  +  3 t^{ 3 p^2 + 3 p^3 }  +  2 t^{ p }  +  3 t^{ p + p^3 }  +  3 t^{ p + 2 p^3 }  +  4 t^{ p + 3 p^3 }  +  2 t^{ p + p^2 }  + t^{ p + p^2 + p^3 }  +  2 t^{ p + p^2 + 2 p^3 }  +  3 t^{ p + p^2 + 3 p^3 }  +  2 t^{ p + 2 p^2 }  +  2 t^{ p + 2 p^2 + p^3 }  +  2 t^{ p + 2 p^2 + 2 p^3 }  +  2 t^{ p + 2 p^2 + 3 p^3 }  +  2 t^{ p + 3 p^2 }  +  3 t^{ p + 3 p^2 + p^3 }  +  3 t^{ p + 3 p^2 + 2 p^3 }  +  2 t^{ p + 3 p^2 + 3 p^3 }  +  4 t^{ p + 4 p^2 + p^3 }  +  3 t^{ p + 4 p^2 + 2 p^3 }  +  4 t^{ p + 4 p^2 + 3 p^3 }  +  4 t^{ 2 p }  +  3 t^{ 2 p + p^3 }  +  2 t^{ 2 p + 2 p^3 }  +  2 t^{ 2 p + p^2 }  +  2 t^{ 2 p + p^2 + p^3 }  +  7 t^{ 2 p + p^2 + 2 p^3 }  +  3 t^{ 2 p + p^2 + 3 p^3 }  +  4 t^{ 2 p + 2 p^2 }  +  2 t^{ 2 p + 2 p^2 + p^3 }  + t^{ 2 p + 2 p^2 + 2 p^3 }  +  2 t^{ 2 p + 2 p^2 + 3 p^3 }  +  3 t^{ 2 p + 3 p^2 }  +  3 t^{ 2 p + 3 p^2 + p^3 }  + t^{ 2 p + 3 p^2 + 2 p^3 }  +  3 t^{ 2 p + 3 p^2 + 3 p^3 }  +  3 t^{ 2 p + 4 p^2 + p^3 }  +  2 t^{ 2 p + 4 p^2 + 2 p^3 }  +  3 t^{ 2 p + 4 p^2 + 3 p^3 }  +  3 t^{ 3 p }  +  4 t^{ 3 p + p^3 }  +  2 t^{ 3 p + p^2 }  +  3 t^{ 3 p + p^2 + p^3 }  +  3 t^{ 3 p + p^2 + 2 p^3 }  +  3 t^{ 3 p + 2 p^2 }  +  2 t^{ 3 p + 2 p^2 + p^3 }  +  2 t^{ 3 p + 2 p^2 + 2 p^3 }  +  4 t^{ 3 p + 2 p^2 + 3 p^3 }  +  3 t^{ 3 p + 3 p^2 }  +  2 t^{ 3 p + 3 p^2 + p^3 }  +  3 t^{ 3 p + 3 p^2 + 2 p^3 }  +  3 t^{ 3 p + 3 p^2 + 3 p^3 }  +  4 t^{ 3 p + 4 p^2 + p^3 }  +  3 t^{ 3 p + 4 p^2 + 2 p^3 }  +  4 t^{ 3 p + 4 p^2 + 3 p^3 }  +  2 t  +  2 t^{ 1 + p^3 }  +  2 t^{ 1 + 2 p^3 }  +  2 t^{ 1 + 3 p^3 }  +  3 t^{ 1 + p^2 }  + t^{ 1 + p^2 + p^3 }  +  2 t^{ 1 + p^2 + 2 p^3 }  +  3 t^{ 1 + p^2 + 3 p^3 }  +  4 t^{ 1 + p^2 + 4 p^3 }  +  3 t^{ 1 + 2 p^2 }  +  2 t^{ 1 + 2 p^2 + p^3 }  +  2 t^{ 1 + 2 p^2 + 2 p^3 }  +  3 t^{ 1 + 2 p^2 + 3 p^3 }  +  3 t^{ 1 + 2 p^2 + 4 p^3 }  +  4 t^{ 1 + 3 p^2 }  +  3 t^{ 1 + 3 p^2 + p^3 }  +  2 t^{ 1 + 3 p^2 + 2 p^3 }  +  2 t^{ 1 + 3 p^2 + 3 p^3 }  +  4 t^{ 1 + 3 p^2 + 4 p^3 }  +  2 t^{ 1 + p }  + t^{ 1 + p + p^3 }  +  2 t^{ 1 + p + 2 p^3 }  +  3 t^{ 1 + p + 3 p^3 }  + t^{ 1 + p + p^2 }  +  2 t^{ 1 + p + 2 p^2 }  +  3 t^{ 1 + p + 3 p^2 }  +  2 t^{ 1 + 2 p }  +  2 t^{ 1 + 2 p + p^3 }  +  7 t^{ 1 + 2 p + 2 p^3 }  +  3 t^{ 1 + 2 p + 3 p^3 }  +  2 t^{ 1 + 2 p + p^2 }  +  2 t^{ 1 + 2 p + 2 p^2 }  +  2 t^{ 1 + 2 p + 3 p^2 }  +  2 t^{ 1 + 3 p }  +  3 t^{ 1 + 3 p + p^3 }  +  3 t^{ 1 + 3 p + 2 p^3 }  +  3 t^{ 1 + 3 p + p^2 }  +  3 t^{ 1 + 3 p + 2 p^2 }  +  2 t^{ 1 + 3 p + 3 p^2 }  +  4 t^{ 1 + 4 p + p^2 }  +  3 t^{ 1 + 4 p + 2 p^2 }  +  4 t^{ 1 + 4 p + 3 p^2 }  +  4 t^{ 2 }  +  2 t^{ 2 + p^3 }  +  4 t^{ 2 + 2 p^3 }  +  3 t^{ 2 + 3 p^3 }  +  3 t^{ 2 + p^2 }  +  2 t^{ 2 + p^2 + p^3 }  +  2 t^{ 2 + p^2 + 2 p^3 }  +  3 t^{ 2 + p^2 + 3 p^3 }  +  3 t^{ 2 + p^2 + 4 p^3 }  +  2 t^{ 2 + 2 p^2 }  +  7 t^{ 2 + 2 p^2 + p^3 }  + t^{ 2 + 2 p^2 + 2 p^3 }  + t^{ 2 + 2 p^2 + 3 p^3 }  +  2 t^{ 2 + 2 p^2 + 4 p^3 }  +  3 t^{ 2 + 3 p^2 + p^3 }  +  2 t^{ 2 + 3 p^2 + 2 p^3 }  +  3 t^{ 2 + 3 p^2 + 3 p^3 }  +  3 t^{ 2 + 3 p^2 + 4 p^3 }  +  2 t^{ 2 + p }  +  2 t^{ 2 + p + p^3 }  +  2 t^{ 2 + p + 2 p^3 }  +  2 t^{ 2 + p + 3 p^3 }  +  2 t^{ 2 + p + p^2 }  +  7 t^{ 2 + p + 2 p^2 }  +  3 t^{ 2 + p + 3 p^2 }  +  4 t^{ 2 + 2 p }  +  2 t^{ 2 + 2 p + p^3 }  + t^{ 2 + 2 p + 2 p^3 }  +  2 t^{ 2 + 2 p + 3 p^3 }  +  2 t^{ 2 + 2 p + p^2 }  + t^{ 2 + 2 p + 2 p^2 }  +  2 t^{ 2 + 2 p + 3 p^2 }  +  3 t^{ 2 + 3 p }  +  2 t^{ 2 + 3 p + p^3 }  +  2 t^{ 2 + 3 p + 2 p^3 }  +  4 t^{ 2 + 3 p + 3 p^3 }  +  3 t^{ 2 + 3 p + p^2 }  + t^{ 2 + 3 p + 2 p^2 }  +  3 t^{ 2 + 3 p + 3 p^2 }  +  3 t^{ 2 + 4 p + p^2 }  +  2 t^{ 2 + 4 p + 2 p^2 }  +  3 t^{ 2 + 4 p + 3 p^2 }  +  3 t^{ 3 }  +  2 t^{ 3 + p^3 }  +  3 t^{ 3 + 2 p^3 }  +  3 t^{ 3 + 3 p^3 }  +  4 t^{ 3 + p^2 }  +  3 t^{ 3 + p^2 + p^3 }  +  2 t^{ 3 + p^2 + 2 p^3 }  +  2 t^{ 3 + p^2 + 3 p^3 }  +  4 t^{ 3 + p^2 + 4 p^3 }  +  3 t^{ 3 + 2 p^2 + p^3 }  +  2 t^{ 3 + 2 p^2 + 2 p^3 }  +  3 t^{ 3 + 2 p^2 + 3 p^3 }  +  3 t^{ 3 + 2 p^2 + 4 p^3 }  +  4 t^{ 3 + 3 p^2 + 2 p^3 }  +  3 t^{ 3 + 3 p^2 + 3 p^3 }  +  4 t^{ 3 + 3 p^2 + 4 p^3 }  +  2 t^{ 3 + p }  +  3 t^{ 3 + p + p^3 }  +  3 t^{ 3 + p + 2 p^3 }  +  2 t^{ 3 + p + 3 p^3 }  +  3 t^{ 3 + p + p^2 }  +  3 t^{ 3 + p + 2 p^2 }  +  3 t^{ 3 + 2 p }  +  3 t^{ 3 + 2 p + p^3 }  + t^{ 3 + 2 p + 2 p^3 }  +  3 t^{ 3 + 2 p + 3 p^3 }  +  2 t^{ 3 + 2 p + p^2 }  +  2 t^{ 3 + 2 p + 2 p^2 }  +  4 t^{ 3 + 2 p + 3 p^2 }  +  3 t^{ 3 + 3 p }  +  2 t^{ 3 + 3 p + p^3 }  +  3 t^{ 3 + 3 p + 2 p^3 }  +  3 t^{ 3 + 3 p + 3 p^3 }  +  2 t^{ 3 + 3 p + p^2 }  +  3 t^{ 3 + 3 p + 2 p^2 }  +  3 t^{ 3 + 3 p + 3 p^2 }  +  4 t^{ 3 + 4 p + p^2 }  +  3 t^{ 3 + 4 p + 2 p^2 }  +  4 t^{ 3 + 4 p + 3 p^2 }  +  4 t^{ 4 + p + p^3 }  +  3 t^{ 4 + p + 2 p^3 }  +  4 t^{ 4 + p + 3 p^3 }  +  3 t^{ 4 + 2 p + p^3 }  +  2 t^{ 4 + 2 p + 2 p^3 }  +  3 t^{ 4 + 2 p + 3 p^3 }  +  4 t^{ 4 + 3 p + p^3 }  +  3 t^{ 4 + 3 p + 2 p^3 }  +  4 t^{ 4 + 3 p + 3 p^3 } \right) \end{dmath*}\begin{dmath*} 
 + s^8 \left(  2  + t^{ p^3 }  +  2 t^{ 2 p^3 }  +  2 t^{ 3 p^3 }  + t^{ p^2 }  +  2 t^{ p^2 + p^3 }  + t^{ p^2 + 2 p^3 }  +  3 t^{ p^2 + 3 p^3 }  +  2 t^{ 2 p^2 }  + t^{ 2 p^2 + p^3 }  +  6 t^{ 2 p^2 + 2 p^3 }  +  3 t^{ 2 p^2 + 3 p^3 }  + t^{ 3 p^2 }  +  3 t^{ 3 p^2 + p^3 }  +  3 t^{ 3 p^2 + 2 p^3 }  +  2 t^{ 3 p^2 + 3 p^3 }  + t^{ p }  +  2 t^{ p + p^3 }  +  5 t^{ p + 2 p^3 }  +  6 t^{ p + 3 p^3 }  +  2 t^{ p + p^2 }  + t^{ p + p^2 + p^3 }  +  4 t^{ p + p^2 + 2 p^3 }  +  3 t^{ p + p^2 + 3 p^3 }  + t^{ p + 2 p^2 }  +  4 t^{ p + 2 p^2 + p^3 }  + t^{ p + 2 p^2 + 2 p^3 }  +  2 t^{ p + 2 p^2 + 3 p^3 }  +  3 t^{ p + 3 p^2 }  +  3 t^{ p + 3 p^2 + p^3 }  +  2 t^{ p + 3 p^2 + 2 p^3 }  +  3 t^{ p + 3 p^2 + 3 p^3 }  +  6 t^{ p + 4 p^2 + p^3 }  +  2 t^{ p + 4 p^2 + 2 p^3 }  +  6 t^{ p + 4 p^2 + 3 p^3 }  +  2 t^{ 2 p }  +  5 t^{ 2 p + p^3 }  +  4 t^{ 2 p + 2 p^3 }  + t^{ 2 p + p^2 }  +  4 t^{ 2 p + p^2 + p^3 }  +  5 t^{ 2 p + p^2 + 2 p^3 }  +  2 t^{ 2 p + p^2 + 3 p^3 }  +  6 t^{ 2 p + 2 p^2 }  + t^{ 2 p + 2 p^2 + p^3 }  +  2 t^{ 2 p + 2 p^2 + 2 p^3 }  +  3 t^{ 2 p + 2 p^2 + 3 p^3 }  +  3 t^{ 2 p + 3 p^2 }  +  2 t^{ 2 p + 3 p^2 + p^3 }  +  3 t^{ 2 p + 3 p^2 + 2 p^3 }  +  3 t^{ 2 p + 3 p^2 + 3 p^3 }  +  2 t^{ 2 p + 4 p^2 + p^3 }  +  4 t^{ 2 p + 4 p^2 + 2 p^3 }  +  2 t^{ 2 p + 4 p^2 + 3 p^3 }  +  2 t^{ 3 p }  +  6 t^{ 3 p + p^3 }  +  3 t^{ 3 p + p^2 }  +  3 t^{ 3 p + p^2 + p^3 }  +  2 t^{ 3 p + p^2 + 2 p^3 }  +  3 t^{ 3 p + 2 p^2 }  +  2 t^{ 3 p + 2 p^2 + p^3 }  +  3 t^{ 3 p + 2 p^2 + 2 p^3 }  +  6 t^{ 3 p + 2 p^2 + 3 p^3 }  +  2 t^{ 3 p + 3 p^2 }  +  3 t^{ 3 p + 3 p^2 + p^3 }  +  3 t^{ 3 p + 3 p^2 + 2 p^3 }  +  2 t^{ 3 p + 3 p^2 + 3 p^3 }  +  6 t^{ 3 p + 4 p^2 + p^3 }  +  2 t^{ 3 p + 4 p^2 + 2 p^3 }  +  6 t^{ 3 p + 4 p^2 + 3 p^3 }  + t  +  2 t^{ 1 + p^3 }  + t^{ 1 + 2 p^3 }  +  3 t^{ 1 + 3 p^3 }  +  2 t^{ 1 + p^2 }  + t^{ 1 + p^2 + p^3 }  +  4 t^{ 1 + p^2 + 2 p^3 }  +  3 t^{ 1 + p^2 + 3 p^3 }  +  6 t^{ 1 + p^2 + 4 p^3 }  +  5 t^{ 1 + 2 p^2 }  +  4 t^{ 1 + 2 p^2 + p^3 }  + t^{ 1 + 2 p^2 + 2 p^3 }  +  2 t^{ 1 + 2 p^2 + 3 p^3 }  +  2 t^{ 1 + 2 p^2 + 4 p^3 }  +  6 t^{ 1 + 3 p^2 }  +  3 t^{ 1 + 3 p^2 + p^3 }  +  2 t^{ 1 + 3 p^2 + 2 p^3 }  +  3 t^{ 1 + 3 p^2 + 3 p^3 }  +  6 t^{ 1 + 3 p^2 + 4 p^3 }  +  2 t^{ 1 + p }  + t^{ 1 + p + p^3 }  +  4 t^{ 1 + p + 2 p^3 }  +  3 t^{ 1 + p + 3 p^3 }  + t^{ 1 + p + p^2 }  +  4 t^{ 1 + p + 2 p^2 }  +  3 t^{ 1 + p + 3 p^2 }  + t^{ 1 + 2 p }  +  4 t^{ 1 + 2 p + p^3 }  +  5 t^{ 1 + 2 p + 2 p^3 }  +  2 t^{ 1 + 2 p + 3 p^3 }  +  4 t^{ 1 + 2 p + p^2 }  + t^{ 1 + 2 p + 2 p^2 }  +  2 t^{ 1 + 2 p + 3 p^2 }  +  3 t^{ 1 + 3 p }  +  3 t^{ 1 + 3 p + p^3 }  +  2 t^{ 1 + 3 p + 2 p^3 }  +  3 t^{ 1 + 3 p + p^2 }  +  2 t^{ 1 + 3 p + 2 p^2 }  +  3 t^{ 1 + 3 p + 3 p^2 }  +  6 t^{ 1 + 4 p + p^2 }  +  2 t^{ 1 + 4 p + 2 p^2 }  +  6 t^{ 1 + 4 p + 3 p^2 }  +  2 t^{ 2 }  + t^{ 2 + p^3 }  +  6 t^{ 2 + 2 p^3 }  +  3 t^{ 2 + 3 p^3 }  +  5 t^{ 2 + p^2 }  +  4 t^{ 2 + p^2 + p^3 }  + t^{ 2 + p^2 + 2 p^3 }  +  2 t^{ 2 + p^2 + 3 p^3 }  + t^{ 2 + p^2 + 4 p^3 }  +  4 t^{ 2 + 2 p^2 }  +  5 t^{ 2 + 2 p^2 + p^3 }  +  2 t^{ 2 + 2 p^2 + 2 p^3 }  +  3 t^{ 2 + 2 p^2 + 3 p^3 }  +  2 t^{ 2 + 2 p^2 + 4 p^3 }  +  2 t^{ 2 + 3 p^2 + p^3 }  +  3 t^{ 2 + 3 p^2 + 2 p^3 }  +  3 t^{ 2 + 3 p^2 + 3 p^3 }  + t^{ 2 + 3 p^2 + 4 p^3 }  + t^{ 2 + p }  +  4 t^{ 2 + p + p^3 }  + t^{ 2 + p + 2 p^3 }  +  2 t^{ 2 + p + 3 p^3 }  +  4 t^{ 2 + p + p^2 }  +  5 t^{ 2 + p + 2 p^2 }  +  2 t^{ 2 + p + 3 p^2 }  +  6 t^{ 2 + 2 p }  + t^{ 2 + 2 p + p^3 }  +  2 t^{ 2 + 2 p + 2 p^3 }  +  3 t^{ 2 + 2 p + 3 p^3 }  + t^{ 2 + 2 p + p^2 }  +  2 t^{ 2 + 2 p + 2 p^2 }  +  3 t^{ 2 + 2 p + 3 p^2 }  +  3 t^{ 2 + 3 p }  +  2 t^{ 2 + 3 p + p^3 }  +  3 t^{ 2 + 3 p + 2 p^3 }  +  6 t^{ 2 + 3 p + 3 p^3 }  +  2 t^{ 2 + 3 p + p^2 }  +  3 t^{ 2 + 3 p + 2 p^2 }  +  3 t^{ 2 + 3 p + 3 p^2 }  + t^{ 2 + 4 p + p^2 }  +  2 t^{ 2 + 4 p + 2 p^2 }  + t^{ 2 + 4 p + 3 p^2 }  +  2 t^{ 3 }  +  3 t^{ 3 + p^3 }  +  3 t^{ 3 + 2 p^3 }  +  2 t^{ 3 + 3 p^3 }  +  6 t^{ 3 + p^2 }  +  3 t^{ 3 + p^2 + p^3 }  +  2 t^{ 3 + p^2 + 2 p^3 }  +  3 t^{ 3 + p^2 + 3 p^3 }  +  6 t^{ 3 + p^2 + 4 p^3 }  +  2 t^{ 3 + 2 p^2 + p^3 }  +  3 t^{ 3 + 2 p^2 + 2 p^3 }  +  3 t^{ 3 + 2 p^2 + 3 p^3 }  +  2 t^{ 3 + 2 p^2 + 4 p^3 }  +  6 t^{ 3 + 3 p^2 + 2 p^3 }  +  2 t^{ 3 + 3 p^2 + 3 p^3 }  +  6 t^{ 3 + 3 p^2 + 4 p^3 }  +  3 t^{ 3 + p }  +  3 t^{ 3 + p + p^3 }  +  2 t^{ 3 + p + 2 p^3 }  +  3 t^{ 3 + p + 3 p^3 }  +  3 t^{ 3 + p + p^2 }  +  2 t^{ 3 + p + 2 p^2 }  +  3 t^{ 3 + 2 p }  +  2 t^{ 3 + 2 p + p^3 }  +  3 t^{ 3 + 2 p + 2 p^3 }  +  3 t^{ 3 + 2 p + 3 p^3 }  +  2 t^{ 3 + 2 p + p^2 }  +  3 t^{ 3 + 2 p + 2 p^2 }  +  6 t^{ 3 + 2 p + 3 p^2 }  +  2 t^{ 3 + 3 p }  +  3 t^{ 3 + 3 p + p^3 }  +  3 t^{ 3 + 3 p + 2 p^3 }  + t^{ 3 + 3 p + 3 p^3 }  +  3 t^{ 3 + 3 p + p^2 }  +  3 t^{ 3 + 3 p + 2 p^2 }  +  2 t^{ 3 + 3 p + 3 p^2 }  +  6 t^{ 3 + 4 p + p^2 }  +  2 t^{ 3 + 4 p + 2 p^2 }  +  6 t^{ 3 + 4 p + 3 p^2 }  +  6 t^{ 4 + p + p^3 }  +  2 t^{ 4 + p + 2 p^3 }  +  6 t^{ 4 + p + 3 p^3 }  +  2 t^{ 4 + 2 p + p^3 }  +  4 t^{ 4 + 2 p + 2 p^3 }  +  2 t^{ 4 + 2 p + 3 p^3 }  +  6 t^{ 4 + 3 p + p^3 }  +  2 t^{ 4 + 3 p + 2 p^3 }  +  6 t^{ 4 + 3 p + 3 p^3 } \right) \end{dmath*}\begin{dmath*} 
 + s^9 \left( 1  + t^{ p^3 }  + t^{ 2 p^3 }  +  3 t^{ 3 p^3 }  + t^{ p^2 }  +  2 t^{ p^2 + p^3 }  +  2 t^{ p^2 + 2 p^3 }  +  3 t^{ p^2 + 3 p^3 }  + t^{ 2 p^2 }  +  2 t^{ 2 p^2 + p^3 }  +  4 t^{ 2 p^2 + 2 p^3 }  +  2 t^{ 2 p^2 + 3 p^3 }  +  3 t^{ 3 p^2 }  +  3 t^{ 3 p^2 + p^3 }  +  2 t^{ 3 p^2 + 2 p^3 }  +  3 t^{ 3 p^2 + 3 p^3 }  + t^{ p }  +  3 t^{ p + p^3 }  +  7 t^{ p + 2 p^3 }  +  4 t^{ p + 3 p^3 }  +  2 t^{ p + p^2 }  +  2 t^{ p + p^2 + p^3 }  +  2 t^{ p + p^2 + 2 p^3 }  +  2 t^{ p + p^2 + 3 p^3 }  +  2 t^{ p + 2 p^2 }  +  2 t^{ p + 2 p^2 + p^3 }  +  2 t^{ p + 2 p^2 + 2 p^3 }  +  2 t^{ p + 2 p^2 + 3 p^3 }  +  3 t^{ p + 3 p^2 }  + t^{ p + 3 p^2 + p^3 }  +  3 t^{ p + 3 p^2 + 2 p^3 }  +  3 t^{ p + 3 p^2 + 3 p^3 }  +  4 t^{ p + 4 p^2 + p^3 }  +  3 t^{ p + 4 p^2 + 2 p^3 }  +  4 t^{ p + 4 p^2 + 3 p^3 }  + t^{ 2 p }  +  7 t^{ 2 p + p^3 }  +  2 t^{ 2 p + 2 p^3 }  +  2 t^{ 2 p + p^2 }  +  2 t^{ 2 p + p^2 + p^3 }  +  3 t^{ 2 p + p^2 + 2 p^3 }  +  3 t^{ 2 p + p^2 + 3 p^3 }  +  4 t^{ 2 p + 2 p^2 }  +  2 t^{ 2 p + 2 p^2 + p^3 }  +  4 t^{ 2 p + 2 p^2 + 2 p^3 }  +  3 t^{ 2 p + 2 p^2 + 3 p^3 }  +  2 t^{ 2 p + 3 p^2 }  +  3 t^{ 2 p + 3 p^2 + p^3 }  +  3 t^{ 2 p + 3 p^2 + 2 p^3 }  +  2 t^{ 2 p + 3 p^2 + 3 p^3 }  +  3 t^{ 2 p + 4 p^2 + p^3 }  +  2 t^{ 2 p + 4 p^2 + 2 p^3 }  +  3 t^{ 2 p + 4 p^2 + 3 p^3 }  +  3 t^{ 3 p }  +  4 t^{ 3 p + p^3 }  +  3 t^{ 3 p + p^2 }  +  2 t^{ 3 p + p^2 + p^3 }  +  3 t^{ 3 p + p^2 + 2 p^3 }  +  2 t^{ 3 p + 2 p^2 }  +  2 t^{ 3 p + 2 p^2 + p^3 }  +  3 t^{ 3 p + 2 p^2 + 2 p^3 }  +  4 t^{ 3 p + 2 p^2 + 3 p^3 }  +  3 t^{ 3 p + 3 p^2 }  +  3 t^{ 3 p + 3 p^2 + p^3 }  +  2 t^{ 3 p + 3 p^2 + 2 p^3 }  +  3 t^{ 3 p + 3 p^2 + 3 p^3 }  +  4 t^{ 3 p + 4 p^2 + p^3 }  +  3 t^{ 3 p + 4 p^2 + 2 p^3 }  +  4 t^{ 3 p + 4 p^2 + 3 p^3 }  + t  +  2 t^{ 1 + p^3 }  +  2 t^{ 1 + 2 p^3 }  +  3 t^{ 1 + 3 p^3 }  +  3 t^{ 1 + p^2 }  +  2 t^{ 1 + p^2 + p^3 }  +  2 t^{ 1 + p^2 + 2 p^3 }  + t^{ 1 + p^2 + 3 p^3 }  +  4 t^{ 1 + p^2 + 4 p^3 }  +  7 t^{ 1 + 2 p^2 }  +  2 t^{ 1 + 2 p^2 + p^3 }  +  2 t^{ 1 + 2 p^2 + 2 p^3 }  +  3 t^{ 1 + 2 p^2 + 3 p^3 }  +  3 t^{ 1 + 2 p^2 + 4 p^3 }  +  4 t^{ 1 + 3 p^2 }  +  2 t^{ 1 + 3 p^2 + p^3 }  +  2 t^{ 1 + 3 p^2 + 2 p^3 }  +  3 t^{ 1 + 3 p^2 + 3 p^3 }  +  4 t^{ 1 + 3 p^2 + 4 p^3 }  +  2 t^{ 1 + p }  +  2 t^{ 1 + p + p^3 }  +  2 t^{ 1 + p + 2 p^3 }  +  2 t^{ 1 + p + 3 p^3 }  +  2 t^{ 1 + p + p^2 }  +  2 t^{ 1 + p + 2 p^2 }  +  2 t^{ 1 + p + 3 p^2 }  +  2 t^{ 1 + 2 p }  +  2 t^{ 1 + 2 p + p^3 }  +  3 t^{ 1 + 2 p + 2 p^3 }  +  3 t^{ 1 + 2 p + 3 p^3 }  +  2 t^{ 1 + 2 p + p^2 }  +  2 t^{ 1 + 2 p + 2 p^2 }  +  2 t^{ 1 + 2 p + 3 p^2 }  +  3 t^{ 1 + 3 p }  +  2 t^{ 1 + 3 p + p^3 }  +  3 t^{ 1 + 3 p + 2 p^3 }  + t^{ 1 + 3 p + p^2 }  +  3 t^{ 1 + 3 p + 2 p^2 }  +  3 t^{ 1 + 3 p + 3 p^2 }  +  4 t^{ 1 + 4 p + p^2 }  +  3 t^{ 1 + 4 p + 2 p^2 }  +  4 t^{ 1 + 4 p + 3 p^2 }  + t^{ 2 }  +  2 t^{ 2 + p^3 }  +  4 t^{ 2 + 2 p^3 }  +  2 t^{ 2 + 3 p^3 }  +  7 t^{ 2 + p^2 }  +  2 t^{ 2 + p^2 + p^3 }  +  2 t^{ 2 + p^2 + 2 p^3 }  +  3 t^{ 2 + p^2 + 3 p^3 }  +  3 t^{ 2 + p^2 + 4 p^3 }  +  2 t^{ 2 + 2 p^2 }  +  3 t^{ 2 + 2 p^2 + p^3 }  +  4 t^{ 2 + 2 p^2 + 2 p^3 }  +  3 t^{ 2 + 2 p^2 + 3 p^3 }  +  2 t^{ 2 + 2 p^2 + 4 p^3 }  +  3 t^{ 2 + 3 p^2 + p^3 }  +  3 t^{ 2 + 3 p^2 + 2 p^3 }  +  2 t^{ 2 + 3 p^2 + 3 p^3 }  +  3 t^{ 2 + 3 p^2 + 4 p^3 }  +  2 t^{ 2 + p }  +  2 t^{ 2 + p + p^3 }  +  2 t^{ 2 + p + 2 p^3 }  +  2 t^{ 2 + p + 3 p^3 }  +  2 t^{ 2 + p + p^2 }  +  3 t^{ 2 + p + 2 p^2 }  +  3 t^{ 2 + p + 3 p^2 }  +  4 t^{ 2 + 2 p }  +  2 t^{ 2 + 2 p + p^3 }  +  4 t^{ 2 + 2 p + 2 p^3 }  +  3 t^{ 2 + 2 p + 3 p^3 }  +  2 t^{ 2 + 2 p + p^2 }  +  4 t^{ 2 + 2 p + 2 p^2 }  +  3 t^{ 2 + 2 p + 3 p^2 }  +  2 t^{ 2 + 3 p }  +  2 t^{ 2 + 3 p + p^3 }  +  3 t^{ 2 + 3 p + 2 p^3 }  +  4 t^{ 2 + 3 p + 3 p^3 }  +  3 t^{ 2 + 3 p + p^2 }  +  3 t^{ 2 + 3 p + 2 p^2 }  +  2 t^{ 2 + 3 p + 3 p^2 }  +  3 t^{ 2 + 4 p + p^2 }  +  2 t^{ 2 + 4 p + 2 p^2 }  +  3 t^{ 2 + 4 p + 3 p^2 }  +  3 t^{ 3 }  +  3 t^{ 3 + p^3 }  +  2 t^{ 3 + 2 p^3 }  +  3 t^{ 3 + 3 p^3 }  +  4 t^{ 3 + p^2 }  +  2 t^{ 3 + p^2 + p^3 }  +  2 t^{ 3 + p^2 + 2 p^3 }  +  3 t^{ 3 + p^2 + 3 p^3 }  +  4 t^{ 3 + p^2 + 4 p^3 }  +  3 t^{ 3 + 2 p^2 + p^3 }  +  3 t^{ 3 + 2 p^2 + 2 p^3 }  +  2 t^{ 3 + 2 p^2 + 3 p^3 }  +  3 t^{ 3 + 2 p^2 + 4 p^3 }  +  4 t^{ 3 + 3 p^2 + 2 p^3 }  +  3 t^{ 3 + 3 p^2 + 3 p^3 }  +  4 t^{ 3 + 3 p^2 + 4 p^3 }  +  3 t^{ 3 + p }  + t^{ 3 + p + p^3 }  +  3 t^{ 3 + p + 2 p^3 }  +  3 t^{ 3 + p + 3 p^3 }  +  2 t^{ 3 + p + p^2 }  +  3 t^{ 3 + p + 2 p^2 }  +  2 t^{ 3 + 2 p }  +  3 t^{ 3 + 2 p + p^3 }  +  3 t^{ 3 + 2 p + 2 p^3 }  +  2 t^{ 3 + 2 p + 3 p^3 }  +  2 t^{ 3 + 2 p + p^2 }  +  3 t^{ 3 + 2 p + 2 p^2 }  +  4 t^{ 3 + 2 p + 3 p^2 }  +  3 t^{ 3 + 3 p }  +  3 t^{ 3 + 3 p + p^3 }  +  2 t^{ 3 + 3 p + 2 p^3 }  +  3 t^{ 3 + 3 p + 3 p^3 }  +  3 t^{ 3 + 3 p + p^2 }  +  2 t^{ 3 + 3 p + 2 p^2 }  +  3 t^{ 3 + 3 p + 3 p^2 }  +  4 t^{ 3 + 4 p + p^2 }  +  3 t^{ 3 + 4 p + 2 p^2 }  +  4 t^{ 3 + 4 p + 3 p^2 }  +  4 t^{ 4 + p + p^3 }  +  3 t^{ 4 + p + 2 p^3 }  +  4 t^{ 4 + p + 3 p^3 }  +  3 t^{ 4 + 2 p + p^3 }  +  2 t^{ 4 + 2 p + 2 p^3 }  +  3 t^{ 4 + 2 p + 3 p^3 }  +  4 t^{ 4 + 3 p + p^3 }  +  3 t^{ 4 + 3 p + 2 p^3 }  +  4 t^{ 4 + 3 p + 3 p^3 } \right) \end{dmath*}\begin{dmath*} 
 + s^{10} \left( 1  +  2 t^{ p^3 }  +  2 t^{ 2 p^3 }  +  3 t^{ 3 p^3 }  +  2 t^{ p^2 }  + t^{ p^2 + p^3 }  + t^{ p^2 + 2 p^3 }  + t^{ p^2 + 3 p^3 }  +  2 t^{ 2 p^2 }  + t^{ 2 p^2 + p^3 }  +  2 t^{ 2 p^2 + 2 p^3 }  +  3 t^{ 2 p^2 + 3 p^3 }  +  3 t^{ 3 p^2 }  + t^{ 3 p^2 + p^3 }  +  3 t^{ 3 p^2 + 2 p^3 }  +  3 t^{ 3 p^2 + 3 p^3 }  +  2 t^{ p }  +  3 t^{ p + p^3 }  +  3 t^{ p + 2 p^3 }  + t^{ p + 3 p^3 }  + t^{ p + p^2 }  +  2 t^{ p + p^2 + p^3 }  + t^{ p + p^2 + 2 p^3 }  +  3 t^{ p + p^2 + 3 p^3 }  + t^{ p + 2 p^2 }  + t^{ p + 2 p^2 + p^3 }  +  4 t^{ p + 2 p^2 + 2 p^3 }  + t^{ p + 2 p^2 + 3 p^3 }  + t^{ p + 3 p^2 }  +  3 t^{ p + 3 p^2 + 2 p^3 }  + t^{ p + 3 p^2 + 3 p^3 }  + t^{ p + 4 p^2 + p^3 }  +  3 t^{ p + 4 p^2 + 2 p^3 }  + t^{ p + 4 p^2 + 3 p^3 }  +  2 t^{ 2 p }  +  3 t^{ 2 p + p^3 }  + t^{ 2 p + 2 p^3 }  + t^{ 2 p + p^2 }  + t^{ 2 p + p^2 + p^3 }  +  5 t^{ 2 p + p^2 + 2 p^3 }  +  3 t^{ 2 p + p^2 + 3 p^3 }  +  2 t^{ 2 p + 2 p^2 }  +  4 t^{ 2 p + 2 p^2 + p^3 }  +  2 t^{ 2 p + 2 p^2 + 2 p^3 }  + t^{ 2 p + 2 p^2 + 3 p^3 }  +  3 t^{ 2 p + 3 p^2 }  +  3 t^{ 2 p + 3 p^2 + p^3 }  +  2 t^{ 2 p + 3 p^2 + 2 p^3 }  +  3 t^{ 2 p + 3 p^2 + 3 p^3 }  +  3 t^{ 2 p + 4 p^2 + p^3 }  + t^{ 2 p + 4 p^2 + 2 p^3 }  +  3 t^{ 2 p + 4 p^2 + 3 p^3 }  +  3 t^{ 3 p }  + t^{ 3 p + p^3 }  + t^{ 3 p + p^2 }  +  3 t^{ 3 p + p^2 + p^3 }  +  3 t^{ 3 p + p^2 + 2 p^3 }  +  3 t^{ 3 p + 2 p^2 }  + t^{ 3 p + 2 p^2 + p^3 }  + t^{ 3 p + 2 p^2 + 2 p^3 }  + t^{ 3 p + 2 p^2 + 3 p^3 }  +  3 t^{ 3 p + 3 p^2 }  + t^{ 3 p + 3 p^2 + p^3 }  +  3 t^{ 3 p + 3 p^2 + 2 p^3 }  +  3 t^{ 3 p + 3 p^2 + 3 p^3 }  + t^{ 3 p + 4 p^2 + p^3 }  +  3 t^{ 3 p + 4 p^2 + 2 p^3 }  + t^{ 3 p + 4 p^2 + 3 p^3 }  +  2 t  + t^{ 1 + p^3 }  + t^{ 1 + 2 p^3 }  + t^{ 1 + 3 p^3 }  +  3 t^{ 1 + p^2 }  +  2 t^{ 1 + p^2 + p^3 }  + t^{ 1 + p^2 + 2 p^3 }  + t^{ 1 + p^2 + 4 p^3 }  +  3 t^{ 1 + 2 p^2 }  + t^{ 1 + 2 p^2 + p^3 }  +  4 t^{ 1 + 2 p^2 + 2 p^3 }  +  3 t^{ 1 + 2 p^2 + 3 p^3 }  +  3 t^{ 1 + 2 p^2 + 4 p^3 }  + t^{ 1 + 3 p^2 }  +  3 t^{ 1 + 3 p^2 + p^3 }  + t^{ 1 + 3 p^2 + 2 p^3 }  + t^{ 1 + 3 p^2 + 3 p^3 }  + t^{ 1 + 3 p^2 + 4 p^3 }  + t^{ 1 + p }  +  2 t^{ 1 + p + p^3 }  + t^{ 1 + p + 2 p^3 }  +  3 t^{ 1 + p + 3 p^3 }  +  2 t^{ 1 + p + p^2 }  + t^{ 1 + p + 2 p^2 }  +  3 t^{ 1 + p + 3 p^2 }  + t^{ 1 + 2 p }  + t^{ 1 + 2 p + p^3 }  +  5 t^{ 1 + 2 p + 2 p^3 }  +  3 t^{ 1 + 2 p + 3 p^3 }  + t^{ 1 + 2 p + p^2 }  +  4 t^{ 1 + 2 p + 2 p^2 }  + t^{ 1 + 2 p + 3 p^2 }  + t^{ 1 + 3 p }  +  3 t^{ 1 + 3 p + p^3 }  +  3 t^{ 1 + 3 p + 2 p^3 }  +  3 t^{ 1 + 3 p + 2 p^2 }  + t^{ 1 + 3 p + 3 p^2 }  + t^{ 1 + 4 p + p^2 }  +  3 t^{ 1 + 4 p + 2 p^2 }  + t^{ 1 + 4 p + 3 p^2 }  +  2 t^{ 2 }  + t^{ 2 + p^3 }  +  2 t^{ 2 + 2 p^3 }  +  3 t^{ 2 + 3 p^3 }  +  3 t^{ 2 + p^2 }  + t^{ 2 + p^2 + p^3 }  +  4 t^{ 2 + p^2 + 2 p^3 }  +  3 t^{ 2 + p^2 + 3 p^3 }  +  3 t^{ 2 + p^2 + 4 p^3 }  + t^{ 2 + 2 p^2 }  +  5 t^{ 2 + 2 p^2 + p^3 }  +  2 t^{ 2 + 2 p^2 + 2 p^3 }  +  2 t^{ 2 + 2 p^2 + 3 p^3 }  + t^{ 2 + 2 p^2 + 4 p^3 }  +  3 t^{ 2 + 3 p^2 + p^3 }  + t^{ 2 + 3 p^2 + 2 p^3 }  +  3 t^{ 2 + 3 p^2 + 3 p^3 }  +  3 t^{ 2 + 3 p^2 + 4 p^3 }  + t^{ 2 + p }  + t^{ 2 + p + p^3 }  +  4 t^{ 2 + p + 2 p^3 }  + t^{ 2 + p + 3 p^3 }  + t^{ 2 + p + p^2 }  +  5 t^{ 2 + p + 2 p^2 }  +  3 t^{ 2 + p + 3 p^2 }  +  2 t^{ 2 + 2 p }  +  4 t^{ 2 + 2 p + p^3 }  +  2 t^{ 2 + 2 p + 2 p^3 }  + t^{ 2 + 2 p + 3 p^3 }  +  4 t^{ 2 + 2 p + p^2 }  +  2 t^{ 2 + 2 p + 2 p^2 }  + t^{ 2 + 2 p + 3 p^2 }  +  3 t^{ 2 + 3 p }  + t^{ 2 + 3 p + p^3 }  + t^{ 2 + 3 p + 2 p^3 }  + t^{ 2 + 3 p + 3 p^3 }  +  3 t^{ 2 + 3 p + p^2 }  +  2 t^{ 2 + 3 p + 2 p^2 }  +  3 t^{ 2 + 3 p + 3 p^2 }  +  3 t^{ 2 + 4 p + p^2 }  + t^{ 2 + 4 p + 2 p^2 }  +  3 t^{ 2 + 4 p + 3 p^2 }  +  3 t^{ 3 }  + t^{ 3 + p^3 }  +  3 t^{ 3 + 2 p^3 }  +  3 t^{ 3 + 3 p^3 }  + t^{ 3 + p^2 }  +  3 t^{ 3 + p^2 + p^3 }  + t^{ 3 + p^2 + 2 p^3 }  + t^{ 3 + p^2 + 3 p^3 }  + t^{ 3 + p^2 + 4 p^3 }  +  3 t^{ 3 + 2 p^2 + p^3 }  + t^{ 3 + 2 p^2 + 2 p^3 }  +  3 t^{ 3 + 2 p^2 + 3 p^3 }  +  3 t^{ 3 + 2 p^2 + 4 p^3 }  + t^{ 3 + 3 p^2 + 2 p^3 }  +  3 t^{ 3 + 3 p^2 + 3 p^3 }  + t^{ 3 + 3 p^2 + 4 p^3 }  + t^{ 3 + p }  +  3 t^{ 3 + p + 2 p^3 }  + t^{ 3 + p + 3 p^3 }  +  3 t^{ 3 + p + p^2 }  +  3 t^{ 3 + p + 2 p^2 }  +  3 t^{ 3 + 2 p }  +  3 t^{ 3 + 2 p + p^3 }  +  2 t^{ 3 + 2 p + 2 p^3 }  +  3 t^{ 3 + 2 p + 3 p^3 }  + t^{ 3 + 2 p + p^2 }  + t^{ 3 + 2 p + 2 p^2 }  + t^{ 3 + 2 p + 3 p^2 }  +  3 t^{ 3 + 3 p }  + t^{ 3 + 3 p + p^3 }  +  3 t^{ 3 + 3 p + 2 p^3 }  +  3 t^{ 3 + 3 p + 3 p^3 }  + t^{ 3 + 3 p + p^2 }  +  3 t^{ 3 + 3 p + 2 p^2 }  +  3 t^{ 3 + 3 p + 3 p^2 }  + t^{ 3 + 4 p + p^2 }  +  3 t^{ 3 + 4 p + 2 p^2 }  + t^{ 3 + 4 p + 3 p^2 }  + t^{ 4 + p + p^3 }  +  3 t^{ 4 + p + 2 p^3 }  + t^{ 4 + p + 3 p^3 }  +  3 t^{ 4 + 2 p + p^3 }  + t^{ 4 + 2 p + 2 p^3 }  +  3 t^{ 4 + 2 p + 3 p^3 }  + t^{ 4 + 3 p + p^3 }  +  3 t^{ 4 + 3 p + 2 p^3 }  + t^{ 4 + 3 p + 3 p^3 } \right)\end{dmath*}\begin{dmath*} 
 + s^{11} \left( 1  + t^{ p^3 }  + t^{ 2 p^3 }  + t^{ 3 p^3 }  + t^{ p^2 }  + t^{ p^2 + p^3 }  + t^{ 2 p^2 }  +  3 t^{ 2 p^2 + 2 p^3 }  +  3 t^{ 2 p^2 + 3 p^3 }  + t^{ 3 p^2 }  +  3 t^{ 3 p^2 + 2 p^3 }  + t^{ 3 p^2 + 3 p^3 }  + t^{ p }  +  2 t^{ p + p^3 }  + t^{ p + 2 p^3 }  + t^{ p + p^2 }  +  2 t^{ p + p^2 + p^3 }  +  2 t^{ p + p^2 + 2 p^3 }  +  3 t^{ p + p^2 + 3 p^3 }  +  2 t^{ p + 2 p^2 + p^3 }  +  2 t^{ p + 2 p^2 + 2 p^3 }  + t^{ p + 2 p^2 + 3 p^3 }  + t^{ p + 3 p^2 + 2 p^3 }  + t^{ p + 4 p^2 + 2 p^3 }  + t^{ 2 p }  + t^{ 2 p + p^3 }  +  2 t^{ 2 p + 2 p^3 }  +  2 t^{ 2 p + p^2 + p^3 }  +  4 t^{ 2 p + p^2 + 2 p^3 }  + t^{ 2 p + p^2 + 3 p^3 }  +  3 t^{ 2 p + 2 p^2 }  +  2 t^{ 2 p + 2 p^2 + p^3 }  + t^{ 2 p + 2 p^2 + 2 p^3 }  +  3 t^{ 2 p + 3 p^2 }  + t^{ 2 p + 3 p^2 + p^3 }  +  3 t^{ 2 p + 3 p^2 + 2 p^3 }  +  3 t^{ 2 p + 3 p^2 + 3 p^3 }  + t^{ 2 p + 4 p^2 + p^3 }  +  2 t^{ 2 p + 4 p^2 + 2 p^3 }  + t^{ 2 p + 4 p^2 + 3 p^3 }  + t^{ 3 p }  +  3 t^{ 3 p + p^2 + p^3 }  + t^{ 3 p + p^2 + 2 p^3 }  +  3 t^{ 3 p + 2 p^2 }  + t^{ 3 p + 2 p^2 + p^3 }  + t^{ 3 p + 3 p^2 }  +  3 t^{ 3 p + 3 p^2 + 2 p^3 }  + t^{ 3 p + 3 p^2 + 3 p^3 }  + t^{ 3 p + 4 p^2 + 2 p^3 }  + t  + t^{ 1 + p^3 }  +  2 t^{ 1 + p^2 }  +  2 t^{ 1 + p^2 + p^3 }  +  2 t^{ 1 + p^2 + 2 p^3 }  + t^{ 1 + 2 p^2 }  +  2 t^{ 1 + 2 p^2 + p^3 }  +  2 t^{ 1 + 2 p^2 + 2 p^3 }  + t^{ 1 + 2 p^2 + 3 p^3 }  + t^{ 1 + 2 p^2 + 4 p^3 }  +  3 t^{ 1 + 3 p^2 + p^3 }  + t^{ 1 + 3 p^2 + 2 p^3 }  + t^{ 1 + p }  +  2 t^{ 1 + p + p^3 }  +  2 t^{ 1 + p + 2 p^3 }  +  3 t^{ 1 + p + 3 p^3 }  +  2 t^{ 1 + p + p^2 }  +  2 t^{ 1 + p + 2 p^2 }  +  3 t^{ 1 + p + 3 p^2 }  +  2 t^{ 1 + 2 p + p^3 }  +  4 t^{ 1 + 2 p + 2 p^3 }  + t^{ 1 + 2 p + 3 p^3 }  +  2 t^{ 1 + 2 p + p^2 }  +  2 t^{ 1 + 2 p + 2 p^2 }  + t^{ 1 + 2 p + 3 p^2 }  +  3 t^{ 1 + 3 p + p^3 }  + t^{ 1 + 3 p + 2 p^3 }  + t^{ 1 + 3 p + 2 p^2 }  + t^{ 1 + 4 p + 2 p^2 }  + t^{ 2 }  +  3 t^{ 2 + 2 p^3 }  +  3 t^{ 2 + 3 p^3 }  + t^{ 2 + p^2 }  +  2 t^{ 2 + p^2 + p^3 }  +  2 t^{ 2 + p^2 + 2 p^3 }  + t^{ 2 + p^2 + 3 p^3 }  + t^{ 2 + p^2 + 4 p^3 }  +  2 t^{ 2 + 2 p^2 }  +  4 t^{ 2 + 2 p^2 + p^3 }  + t^{ 2 + 2 p^2 + 2 p^3 }  +  3 t^{ 2 + 2 p^2 + 3 p^3 }  +  2 t^{ 2 + 2 p^2 + 4 p^3 }  + t^{ 2 + 3 p^2 + p^3 }  +  3 t^{ 2 + 3 p^2 + 3 p^3 }  + t^{ 2 + 3 p^2 + 4 p^3 }  +  2 t^{ 2 + p + p^3 }  +  2 t^{ 2 + p + 2 p^3 }  + t^{ 2 + p + 3 p^3 }  +  2 t^{ 2 + p + p^2 }  +  4 t^{ 2 + p + 2 p^2 }  + t^{ 2 + p + 3 p^2 }  +  3 t^{ 2 + 2 p }  +  2 t^{ 2 + 2 p + p^3 }  + t^{ 2 + 2 p + 2 p^3 }  +  2 t^{ 2 + 2 p + p^2 }  + t^{ 2 + 2 p + 2 p^2 }  +  3 t^{ 2 + 3 p }  + t^{ 2 + 3 p + p^3 }  + t^{ 2 + 3 p + p^2 }  +  3 t^{ 2 + 3 p + 2 p^2 }  +  3 t^{ 2 + 3 p + 3 p^2 }  + t^{ 2 + 4 p + p^2 }  +  2 t^{ 2 + 4 p + 2 p^2 }  + t^{ 2 + 4 p + 3 p^2 }  + t^{ 3 }  +  3 t^{ 3 + 2 p^3 }  + t^{ 3 + 3 p^3 }  +  3 t^{ 3 + p^2 + p^3 }  + t^{ 3 + p^2 + 2 p^3 }  + t^{ 3 + 2 p^2 + p^3 }  +  3 t^{ 3 + 2 p^2 + 3 p^3 }  + t^{ 3 + 2 p^2 + 4 p^3 }  + t^{ 3 + 3 p^2 + 3 p^3 }  + t^{ 3 + p + 2 p^3 }  +  3 t^{ 3 + p + p^2 }  + t^{ 3 + p + 2 p^2 }  +  3 t^{ 3 + 2 p }  + t^{ 3 + 2 p + p^3 }  +  3 t^{ 3 + 2 p + 2 p^3 }  +  3 t^{ 3 + 2 p + 3 p^3 }  + t^{ 3 + 2 p + p^2 }  + t^{ 3 + 3 p }  +  3 t^{ 3 + 3 p + 2 p^3 }  + t^{ 3 + 3 p + 3 p^3 }  +  3 t^{ 3 + 3 p + 2 p^2 }  + t^{ 3 + 3 p + 3 p^2 }  + t^{ 3 + 4 p + 2 p^2 }  + t^{ 4 + p + 2 p^3 }  + t^{ 4 + 2 p + p^3 }  +  2 t^{ 4 + 2 p + 2 p^3 }  + t^{ 4 + 2 p + 3 p^3 }  + t^{ 4 + 3 p + 2 p^3 } \right)
 + s^{12} \left( 1  +  2 t^{ p^2 + p^3 }  +  3 t^{ 2 p^2 + 2 p^3 }  + t^{ 2 p^2 + 3 p^3 }  + t^{ 3 p^2 + 2 p^3 }  +  3 t^{ p + p^3 }  +  2 t^{ p + 2 p^3 }  +  2 t^{ p + p^2 }  + t^{ p + p^2 + p^3 }  + t^{ p + p^2 + 2 p^3 }  + t^{ p + p^2 + 3 p^3 }  + t^{ p + 2 p^2 + p^3 }  + t^{ p + 2 p^2 + 2 p^3 }  +  2 t^{ p + 2 p^2 + 3 p^3 }  +  2 t^{ 2 p + p^3 }  + t^{ 2 p + 2 p^3 }  + t^{ 2 p + p^2 + p^3 }  + t^{ 2 p + p^2 + 2 p^3 }  +  3 t^{ 2 p + 2 p^2 }  + t^{ 2 p + 2 p^2 + p^3 }  +  2 t^{ 2 p + 2 p^2 + 2 p^3 }  + t^{ 2 p + 3 p^2 }  +  3 t^{ 2 p + 3 p^2 + 2 p^3 }  + t^{ 2 p + 3 p^2 + 3 p^3 }  + t^{ 2 p + 4 p^2 + 2 p^3 }  + t^{ 3 p + p^2 + p^3 }  + t^{ 3 p + 2 p^2 }  +  2 t^{ 3 p + 2 p^2 + p^3 }  + t^{ 3 p + 3 p^2 + 2 p^3 }  +  2 t^{ 1 + p^3 }  +  3 t^{ 1 + p^2 }  + t^{ 1 + p^2 + p^3 }  + t^{ 1 + p^2 + 2 p^3 }  +  2 t^{ 1 + 2 p^2 }  + t^{ 1 + 2 p^2 + p^3 }  + t^{ 1 + 2 p^2 + 2 p^3 }  + t^{ 1 + 3 p^2 + p^3 }  +  2 t^{ 1 + 3 p^2 + 2 p^3 }  +  2 t^{ 1 + p }  + t^{ 1 + p + p^3 }  + t^{ 1 + p + 2 p^3 }  + t^{ 1 + p + 3 p^3 }  + t^{ 1 + p + p^2 }  + t^{ 1 + p + 2 p^2 }  + t^{ 1 + p + 3 p^2 }  + t^{ 1 + 2 p + p^3 }  + t^{ 1 + 2 p + 2 p^3 }  + t^{ 1 + 2 p + p^2 }  + t^{ 1 + 2 p + 2 p^2 }  +  2 t^{ 1 + 2 p + 3 p^2 }  + t^{ 1 + 3 p + p^3 }  +  3 t^{ 2 + 2 p^3 }  + t^{ 2 + 3 p^3 }  +  2 t^{ 2 + p^2 }  + t^{ 2 + p^2 + p^3 }  + t^{ 2 + p^2 + 2 p^3 }  + t^{ 2 + 2 p^2 }  + t^{ 2 + 2 p^2 + p^3 }  +  2 t^{ 2 + 2 p^2 + 2 p^3 }  +  3 t^{ 2 + 2 p^2 + 3 p^3 }  + t^{ 2 + 2 p^2 + 4 p^3 }  + t^{ 2 + 3 p^2 + 3 p^3 }  + t^{ 2 + p + p^3 }  + t^{ 2 + p + 2 p^3 }  +  2 t^{ 2 + p + 3 p^3 }  + t^{ 2 + p + p^2 }  + t^{ 2 + p + 2 p^2 }  +  3 t^{ 2 + 2 p }  + t^{ 2 + 2 p + p^3 }  +  2 t^{ 2 + 2 p + 2 p^3 }  + t^{ 2 + 2 p + p^2 }  +  2 t^{ 2 + 2 p + 2 p^2 }  + t^{ 2 + 3 p }  +  2 t^{ 2 + 3 p + p^3 }  +  3 t^{ 2 + 3 p + 2 p^2 }  + t^{ 2 + 3 p + 3 p^2 }  + t^{ 2 + 4 p + 2 p^2 }  + t^{ 3 + 2 p^3 }  + t^{ 3 + p^2 + p^3 }  +  2 t^{ 3 + p^2 + 2 p^3 }  + t^{ 3 + 2 p^2 + 3 p^3 }  + t^{ 3 + p + p^2 }  + t^{ 3 + 2 p }  +  3 t^{ 3 + 2 p + 2 p^3 }  + t^{ 3 + 2 p + 3 p^3 }  +  2 t^{ 3 + 2 p + p^2 }  + t^{ 3 + 3 p + 2 p^3 }  + t^{ 3 + 3 p + 2 p^2 }  + t^{ 4 + 2 p + 2 p^3 } \right)
 + s^{13} \left( 1  + t^{ p^2 + p^3 }  + t^{ 2 p^2 + 2 p^3 }  +  3 t^{ p + p^3 }  + t^{ p + 2 p^3 }  + t^{ p + p^2 }  + t^{ p + p^2 + p^3 }  +  2 t^{ p + 2 p^2 + 2 p^3 }  + t^{ p + 2 p^2 + 3 p^3 }  + t^{ 2 p + p^3 }  + t^{ 2 p + 2 p^2 }  +  2 t^{ 2 p + 2 p^2 + p^3 }  + t^{ 2 p + 2 p^2 + 2 p^3 }  + t^{ 2 p + 3 p^2 + 2 p^3 }  + t^{ 3 p + 2 p^2 + p^3 }  + t^{ 1 + p^3 }  +  3 t^{ 1 + p^2 }  + t^{ 1 + p^2 + p^3 }  + t^{ 1 + 2 p^2 }  +  2 t^{ 1 + 2 p^2 + 2 p^3 }  + t^{ 1 + 3 p^2 + 2 p^3 }  + t^{ 1 + p }  + t^{ 1 + p + p^3 }  + t^{ 1 + p + p^2 }  +  2 t^{ 1 + 2 p + 2 p^2 }  + t^{ 1 + 2 p + 3 p^2 }  + t^{ 2 + 2 p^3 }  + t^{ 2 + p^2 }  +  2 t^{ 2 + p^2 + 2 p^3 }  + t^{ 2 + 2 p^2 + 2 p^3 }  + t^{ 2 + 2 p^2 + 3 p^3 }  +  2 t^{ 2 + p + 2 p^3 }  + t^{ 2 + p + 3 p^3 }  + t^{ 2 + 2 p }  +  2 t^{ 2 + 2 p + p^3 }  + t^{ 2 + 2 p + 2 p^3 }  +  2 t^{ 2 + 2 p + p^2 }  + t^{ 2 + 2 p + 2 p^2 }  + t^{ 2 + 3 p + p^3 }  + t^{ 2 + 3 p + 2 p^2 }  + t^{ 3 + p^2 + 2 p^3 }  + t^{ 3 + 2 p + 2 p^3 }  + t^{ 3 + 2 p + p^2 } \right)
 + s^{14} \left( t^{ p + p^3 }  +  2 t^{ p + p^2 + p^3 }  + t^{ p + 2 p^2 + 2 p^3 }  + t^{ 2 p + 2 p^2 + p^3 }  + t^{ 1 + p^2 }  +  2 t^{ 1 + p^2 + p^3 }  + t^{ 1 + 2 p^2 + 2 p^3 }  +  2 t^{ 1 + p + p^3 }  +  2 t^{ 1 + p + p^2 }  + t^{ 1 + 2 p + 2 p^2 }  + t^{ 2 + p^2 + 2 p^3 }  + t^{ 2 + p + 2 p^3 }  + t^{ 2 + 2 p + p^3 }  + t^{ 2 + 2 p + p^2 } \right)
 + s^{15} \left( 1  + t^{ p + p^2 + p^3 }  + t^{ 1 + p^2 + p^3 }  + t^{ 1 + p + p^3 }  + t^{ 1 + p + p^2 } \right)
 + s^{16}
\end{dmath*} 
\end{itemize}

The total $\mathbb{F}_p$-vector space dimension of the mod $p$ cohomology of the height $0,1,2,3,$ and $4$ Morava stabilizer group schemes, for $p>n+1$, is $1,2,12,152,$ and $3440$, respectively. These numbers fit into a distinct recursive pattern which bears some kind of resemblance to the way that cohomology of large-height Morava stabilizer group schemes is assembled from cohomology of lower-height Morava stabilizer group schemes, as in Strategy~\ref{main strategy}. I do not know if there is a connection between the numerical phenomenon and the cohomological phenomenon.

The cohomology $H^*(\strictAut(\mathbb{G}_{1/4}); \mathbb{F}_p)$, tensored with $\mathbb{F}_p[v_4^{\pm 1}]$, is the input for the $E(4)$-Adams spectral sequence computing $\pi_*(L_{E(4)}V(3))$, or equivalently (since $V(3)$ is $E(3)$-acyclic), $\pi_*(L_{K(4)}V(3))$. There is no room for nonzero differentials in this spectral sequence, and consequently $H^*(\strictAut(\mathbb{G}_{1/4}); \mathbb{F}_p)\otimes_{\mathbb{F}_p} \mathbb{F}_p[v_4^{\pm 1}]$ is isomorphic, after an appropriate regrading, to $\pi_*(L_{K(4)}V(3))$. At primes less than seven, $V(3)$ is known not to exist, so the computations in this document recover $\pi_*(L_{K(4)}V(3))$ at all primes at which $V(3)$ exists.
I have been told that M. Mahowald conjectured that $V(3)$ is the last Smith-Toda complex to exist, that is, $V(n)$ does not exist for $n>3$ and for all primes; if Mahowald's conjecture is true, then the computations in this document finish the problem of computing the $K(n)$-local stable homotopy groups of the Smith-Toda complex $V(n-1)$.

Although this document focuses on the cohomology of the Morava stabilizer group schemes at large primes, the methods in this document can be applied to the case $p = n+1$, in which case $H^*(\strictAut(\mathbb{G}_{1/n}); \mathbb{F}_p)$ is known to be infinite-dimensional. The computations when $p=n+1$ are more difficult, partly because the Lie-May spectral sequences~\ref{may 1 ss} and~\ref{relative may 1 ss} have nonzero differentials, and partly because the algebra is slightly different: one needs to use formal groups with complex multiplication by $\hat{\mathbb{Z}}_p\left[ \zeta_p\right] = \hat{\mathbb{Z}}_p\left[ \sqrt{-p}\right]$, rather than $\hat{\mathbb{Z}}_p\left[ \sqrt{p}\right]$, and this leads to a slightly different filtration of $\mathcal{K}(2n,2n)$ from the ``arithmetic filtration'' defined in Definition~\ref{def of K dga}.

This project is a continuation of the sequence of ideas found in chapter 6 of Ravenel's~\cite{MR860042}, in which the ideas of Milnor and Moore from~\cite{MR0174052} and of May from~\cite{MR2614527} are put to use to compute the cohomology of Morava stabilizer groups. 
The profound influence of D.C. Ravenel on this work should be 
unmistakable. I am grateful to Ravenel for his generous guidance and help.

\begin{convention}
All formal groups in this document are assumed commutative and one-dimensional.
\end{convention}

\section{Setting up the spectral sequences.}

The following sequence of integers was defined in Theorem~6.3.1 of~\cite{MR860042}.
\begin{definition}{\bf (Ravenel's numbers.)} \label{ravenel numbers} Fix a prime number $p$ and a positive integer $n$. Let $d_{n,i}$ be the integer defined by the formula
\[ d_{n,i} = \left\{ \begin{array}{ll} 0 & \mbox{\ if\ } i\leq 0 \\ \max\{ i,pd_{n,i-n}\} & \mbox{\ if\ } i>0.\end{array} \right. \]
(Clearly $d_{n,i}$ depends on the prime number $p$, but the choice of prime $p$ is suppressed from the notation for $d_{n,i}$.)
\end{definition}

The differential graded algebras we consider all admit an action by $C_n$. We fix a generator for $C_n$ and we write $\sigma$ for this generator.

\begin{definition}\label{def of K dga}
Let $p$ be a prime, $m$ a positive integer, $n$ a positive integer, and let $\mathcal{K}(n,m)$ denote the $\mathbb{Z}$-graded, $\mathbb{Z}/2(p^{n}-1)\mathbb{Z}$-graded,
$C_n$-equivariant differential graded $\mathbb{F}_p$-algebra 
\[\mathcal{K}(n,m) = \Lambda\left(h_{i,j}: 1\leq i\leq m, 0\leq j< n\right) \]
with differential
\[ d(h_{i,j}) = \sum_{k=1}^{i-1} h_{k,j} h_{i-k,j+k},\]
with the convention that $h_{i,k+n} = h_{i,k}$.
The gradings and $C_n$-action are as follows:
\begin{itemize}
\item
the cohomological grading is given by $\left| h_{i,j}\right| = 1$, 
\item 
the $\mathbb{Z}$-grading other than the cohomological $\mathbb{Z}$-grading is called the ``Ravenel grading,'' and is given by $\left| \left| h_{i,j}\right| \right| = d_{n,i}$ (the numbers $d_{n,i}$ are defined in Definition~\ref{ravenel numbers}),
\item the $\mathbb{Z}/2(p^{n}-1)\mathbb{Z}$-grading is called the ``internal'' (also called ``topological'') grading, and 
is given by $\left|\left|\left| h_{i,j} \right|\right| \right| = 2(p^i-1)p^j$,
\item and the $C_n$-action is given by $\sigma(h_{i,j}) = h_{i,j+1}$. Note that the $C_n$-action preserves the cohomological gradings and the Ravenel grading, but not the internal grading; this behavior will be typical in all of the multigraded DGAs we consider, and we adopt the convention that, whenever we speak of a ``multigraded equivariant DGA,'' we assume that the group action preserves all of the gradings except possibly the internal grading.
\end{itemize}

If $n$ is even, then filter $\mathcal{K}(n,m)$ by powers of the ideal $I_{n,m} = (
h_{i,j} - h_{i,j+n/2}: 1\leq i\leq m, 0\leq j<n)$.
We will write $E_0\mathcal{K}(n,m
)$ for the associated graded differential graded algebra of this filtration on $\mathcal{K}(n,m
)$.
\end{definition}
One checks easily that the ideal $I_{n,m}
$ is a $C_n$-equivariant differential graded Hopf ideal which is generated by elements that are homogeneous in both the Ravenel grading and the internal grading. Consequently $E_0\mathcal{K}(n,m
)$ is a bigraded $C_n$-equivariant $\mathbb{Z}/2(p^{n/2}-1)\mathbb{Z}$-graded differential graded $\mathbb{F}_p$-algebra, i.e.,  $E_0\mathcal{K}(n,m
)$ has a total of four gradings. (The reader who is unhappy with a quad-graded equivariant differential algebra will perhaps be reassured to know that any of these gradings, as well as the $C_n$-action, can be safely ignored when reading this document, as long as one is not trying to reproduce the spectral sequence calculations; the only ways in which all these gradings are used is to eliminate possible differentials in various spectral sequences later on in the document, by observing that the spectral sequence differentials respect these gradings. The $C_n$-action is similarly only used to reduce the amount of work in spectral sequence computations, by observing that certain spectral sequence differentials commute with the $C_n$-action.)
\begin{convention}
We adopt the convention that, given an element $x\in \mathcal{K}(n,m
)$,
\begin{itemize}
\item
$\left| x\right|$ denotes the cohomological degree,
\item  $\left|\left| x \right|\right|$ denotes the Ravenel degree,
\item and $\left|\left|\left| x \right|\right|\right|$ denotes the internal degree (that is, the degree in the $\mathbb{Z}/2(p^{n/2}-1)\mathbb{Z}$-grading).
\item When we need to describe all three gradings at once, we will say that $x$ is in tridegree $\left( \left|x\right|, \left|\left|x\right|\right|, \left|\left|\left|x\right|\right|\right|\right)$.
\item If $x$ is instead an element of $E_0\mathcal{K}(n,m)$, then we will use $\left| x\right|, \left|\left| x\right|\right|, \left|\left|\left| x\right|\right|\right|$ as above, and we will write $\left|\left|\left|\left| x\right|\right|\right|\right|$ for the arithmetic degree, i.e., the grading on $E_0\mathcal{K}(n,m)$ induced by the filtration by powers of $I$ on $\mathcal{K}(n,m)$. We then say that $x$ is in quad-degree $\left( \left|x\right|, \left|\left|x\right|\right|, \left|\left|\left|x\right|\right|\right|, \left|\left|\left|\left| x\right|\right|\right|\right|\right)$.
\item We also write ``graded DGA'' to mean a DGA with one grading {\em in addition to} its cohomological grading, `bigraded DGA'' to mean a DGA with two gradings {\em in addition to} its cohomological grading, and so on.
\end{itemize}
\end{convention}

\begin{prop}\label{presentation for associated graded}
The $C_n$-equivariant trigraded DGA $E_0\mathcal{K}(n,m
)$ is isomorphic to the exterior algebra
\[ \Lambda\left( h_{i,j}, w_{i,j^{\prime}}: 1\leq i\leq m, 0\leq j < n/2, 0\leq j^{\prime} \leq n/2 \right) ,\] 
with $h_{i,j}$ in quad-degree $(1,d_{n,i},2(p^i-1)p^j,0)$, with $w_{i,j^{\prime}}$ in tridegree $(1,d_{n,i},2(p^i-1)p^j,1)$,
with $C_n$-action given by $\sigma(h_{i,j}) = h_{i,j+1}$ and $\sigma(w_{i,j}) = w_{i,j+1}$, and with differential given by
\begin{align*}
 d(h_{i,j}) &= \sum_{k=1}^{i-1} h_{k,j} h_{i-k,j+k},\\ 
 d(w_{i,j}) &= \sum_{k=1}^{i-1} \left( 
     w_{k,j} h_{i-k,j+k} + 
     h_{k,j}w_{i-k,j+k}\right)
\end{align*}
with the convention that $h_{i,j+n/2} = 
h_{i,j}$ and $w_{i,j+n/2} = -
w_{i,j}$.
\end{prop}
\begin{proof}
Let $h_{i,j}\in E_0\mathcal{K}(n,m
)$ be the residue class of $h_{i,j}\in \mathcal{K}(n,m)$ modulo $I_{n,m}$, and let 
$w_{i,j}\in E_0\mathcal{K}(n,m
)$ be the residue class of $
h_{i,j}-h_{i,j+n/2} \in I_{n,m}
\subseteq \mathcal{K}(n,m)$ modulo $I_{n,m}^2$.
Then it is a routine computation to arrive at the given presentation for $E_0\mathcal{K}(n,m
)$.
\end{proof}

Recall that, given a formal group $\mathbb{G}$, the {\em strict automorphism group scheme of $\mathbb{G}$,} which I will write $\strictAut(\mathbb{G})$, is defined as the group scheme of automorphisms of $\mathbb{G}$ whose effect on the tangent space of $\mathbb{G}$ is the identity. In more concrete terms: given a commutative ring $R$ and a formal group law $F(X,Y)\in R[[X,Y]]$, an automorphism of $F$ is given by a power series $f(X)\in R[[X]]$, and we say that $f$ is strict if $f(X)\equiv X\mod X^2$. When $\mathbb{G}$ is defined over a finite field of characteristic $p$ and has $p$-height $n$, then $\strictAut(\mathbb{G})$ is sometimes called the {\em height $n$ strict Morava stabilizer group scheme.}
Reduction modulo the quadratic and higher terms gives us a short exact sequence of profinite group schemes
\[ 1 \rightarrow \strictAut(\mathbb{G}) \rightarrow \Aut(\mathbb{G}) \rightarrow \mathbb{G}_m \rightarrow 1\]
and since $R^{\times}$ is usually relatively understandable (e.g. when $R$ is a finite field), when some knowledge of the cohomology of $\Aut(\mathbb{G})$ is required, it is usually possible to arrive at that knowledge via the Lyndon-Hochschild-Serre spectral sequence, if one can compute the cohomology of $\strictAut(\mathbb{G})$.

\begin{theorem}{\bf (Ravenel.)}\label{ravenels may spectral sequences}
Let $\mathbb{G}$ be the formal group law of height $n$ over $\mathbb{F}_p$
which is classified by the map 
$BP_* \rightarrow \mathbb{F}_p$ sending $v_n$ to $1$ and sending $v_i$ to $0$ for all $i\neq n$.
Then there exist $C_n$-equivariant spectral sequences
\begin{align} 
\label{may 1 ss} E_2^{s,t,u,v} \cong H^{s,t,u}\left(\mathcal{K}(n,\floor{\frac{np}{p-1}})\right) \otimes_{\mathbb{F}_p} P\left(b_{i,j}: 1\leq i \leq \floor{\frac{n}{p-1}}, 0\leq j< n\right) &\Rightarrow 
\Cotor_{E_0\mathbb{F}_p[\strictAut(\mathbb{G})]^*}^{s,t,u}(\mathbb{F}_p,\mathbb{F}_p) \\
\nonumber d_r: E_r^{s,t,u,v} &\rightarrow E_r^{s+1,t,u,v+r-1} \\
\label{may 2 ss} E_1^{s,t,u} \cong \Cotor_{E_0\mathbb{F}_p[\strictAut(\mathbb{G})]^*}^{s,t,u}(\mathbb{F}_p,\mathbb{F}_p) &\Rightarrow H^s(\strictAut(\mathbb{G}); \mathbb{F}_p)\\
\nonumber d_r: E_r^{s,t,u} &\rightarrow E_r^{s+1,t-r,u}  ,\end{align}
where:
\begin{itemize}
\item $\floor{x}$ denotes the integer floor of $x$,
\item $b_{i,j}$ is in quad-degree $(2,pd_{n,i}, 2(p^i-1)p^{j+1},1)$,
\item $\sigma(b_{i,j}) = b_{i,j+1}$, 
\item $\mathbb{F}_p[\strictAut(\mathbb{G})]^*$ is the continuous $\mathbb{F}_p$-linear dual of the profinite group ring $\mathbb{F}_p[\strictAut(\mathbb{G})]$, 
\item
and where $E_0\mathbb{F}_p[\strictAut(\mathbb{G})]^*$ is the associated graded Hopf algebra of Ravenel's filtration (see~6.3.1 of~\cite{MR860042}) on $\mathbb{F}_p[\strictAut(\mathbb{G})]^*$.\end{itemize}
\end{theorem}
\begin{proof}
Theorem~6.3.4 of~\cite{MR860042}. The essential points here are:
\begin{itemize}
\item Using the isomorphism $\Cotor_{\mathbb{F}_p[\strictAut(\mathbb{G})]^*}^{*}(\mathbb{F}_p,\mathbb{F}_p) \cong H^*(\strictAut(\mathbb{G}); \mathbb{F}_p)$, 
spectral sequence~\ref{may 2 ss} is simply the spectral sequence of Ravenel's filtration of the Hopf algebra $\mathbb{F}_p[\strictAut(\mathbb{G})]^*$, analogous to J. P. May's famous spectral sequence, from~\cite{MR2614527}, computing the cohomology $\Ext^*_{A}(\mathbb{F}_p,\mathbb{F}_p)$ of the Steenrod algebra $A$ by filtering $A$ by powers of its augmentation ideal. Ravenel's filtration is the increasing filtration on 
\begin{align*} 
 \mathbb{F}_p[\strictAut(\mathbb{G})]^* 
  &\cong \mathbb{F}_p\otimes_{BP_*}BP_*BP\otimes_{BP_*}\mathbb{F}_p  \\
  &\cong \mathbb{F}_p[t_1, t_2, \dots ]/\left( t_i^{p^n} - t_i\mbox{\ \ for\ all\ } i\right)\end{align*}
given by letting 
the element $t_i^{j}$ be in filtration degree $s_p(j)d_{n,i}$, where $s_p(j)$ is the sum of the digits in the base $p$ expansion of $j$, and $d_{n,i}$ is as in Definition~\ref{ravenel numbers}. 
\item Spectral sequence~\ref{may 1 ss} is a special case of the lesser-known spectral sequence from J. P. May's thesis~\cite{MR2614527}, which computes the cohomology of a restricted Lie algebra from the cohomology of its underlying unrestricted Lie algebra. I have been calling this type of spectral sequence a ``Lie-May spectral sequence,'' to distinguish it from the spectral sequence of a filtered Hopf algebra, also introduced in May's thesis (and of which~\ref{may 1 ss} is a special case). 

The point here is that $\mathbb{F}_p[\strictAut(\mathbb{G})]$ is not primitively generated, but $E^0\mathbb{F}_p[\strictAut(\mathbb{G})]$ is; and from~\cite{MR0174052}, the cohomology of a primitively generated Hopf algebra over a finite field is isomorphic to the cohomology of its restricted Lie algebra of primitives.
In Theorem~6.3.4 of~\cite{MR860042} Ravenel shows that the resulting Lie-May spectral sequence splits as a tensor product of a large tensor factor with trivial $E_{\infty}$-term with a tensor factor whose $E_2$-term is the cohomology of a certain finite-dimensional solvable Lie algebra, tensored with 
$P(b_{i,j}: 1\leq i \leq \floor{\frac{n}{p-1}}, 0\leq j< n)$;
the Koszul complex of that finite-dimensional solvable Lie algebra is $\mathcal{K}(n,\floor{\frac{np}{p-1}})$.
\end{itemize}
\end{proof}

We will need a ``relative'' version of Theorem~\ref{ravenels may spectral sequences} in which we compare the cohomology of the automorphism group scheme of a formal group with complex multiplication to the automorphism group scheme of its underlying formal group.
This ``relative'' theorem is Theorem~\ref{relative ravenel-may ss}, but first, here are the necessary definitions:
\begin{definition}\label{def of formal module}
Let $A$ be a commutative ring, and let $R$ be a commutative $A$-algebra. A (one-dimensional) {\em formal group over $R$ with complex multiplication by $A$,} or for short a {\em formal $A$-module over $R$}, is a formal group law $F(X,Y)\in R[[X,Y]]$ 
together with a group homomorphism $\rho: A \rightarrow \Aut(F)$ such that
$\rho(a) \equiv aX \mod X^2$.

If $F,G$ are formal $A$-modules over $R$, a morphism $f: F\rightarrow G$ of formal $A$-modules is a power series $f(X) \in R[[X]]$ such that $f$ is a homomorphism of formal group laws (i.e., $f(F(X,Y)) = G(f(X),f(Y))$), and such that
$f(a)(\rho_F(X)) = \rho_G(f(a)(X))$ for all $a\in A$.
\end{definition}

The classical results on $p$-height and $p$-typicality (as in~\cite{MR0393050}) were generalized to formal $A$-modules, for $A$ a discrete valuation ring (the second and third claims are proven by M. Hazewinkel in~\cite{MR2987372}; the first claim is easier, but a proof can be found in~\cite{cmah2}):
\begin{theorem}
Let $A$ be a discrete valuation ring of characteristic zero, with finite residue field. Then the classifying Hopf algebroid $(L^A,L^AB)$ of formal $A$-modules admits a retract $(V^A,V^AT)$ with the following properties:
\begin{itemize}
\item The inclusion map
$(V^A,V^AT) \hookrightarrow (L^A,L^AB)$
and the retraction map $(L^A,L^AB) \hookrightarrow (V^A,V^AT)$
are maps of graded Hopf algebroids, and are
mutually homotopy-inverse.
\item If $F$ is a formal $A$-module over a commutative $A$-algebra $R$
and the underlying formal group law of $F$ admits a logarithm $\log_F(X)$, 
then the classifying map $L^A \rightarrow R$ factors through the retraction map
$L^A \rightarrow V^A$ if and only if $\log_F(X) = \sum_{n\geq 1} \alpha_n X^{q^n}$
for some $\alpha_1, \alpha_2, \dots \in R\otimes_{\mathbb{Z}} \mathbb{Q}$,
where $q$ is the cardinality of the residue field of $A$.
\item $V^A \cong A[v_1^A, v_2^A, \dots ]$ with $v_n^A$ in grading degree $2(q^n-1)$.
\end{itemize}
\end{theorem}

\begin{definition}\label{def of typicality and height}
Let $A$ be a discrete valuation ring of characteristic zero, with uniformizer $\pi$, and with finite residue field. Let $R$ be a commutative $A$-algebra, and let $\mathbb{G}$ be a formal $A$-module over $R$.
\begin{itemize}
\item
We say that $\mathbb{G}$ is {\em $A$-typical} if the classifying map $L^A \rightarrow R$ factors through the retraction $L^A \rightarrow V^A$. 
\item 
If $\mathbb{G}$ is $A$-typical and $n$ is a nonnegative integer, we say that $\mathbb{G}$ has {\em $A$-height $\geq n$} if the classifying map $V^A \rightarrow R$ factors through the quotient map $V^A \rightarrow V^A/(\pi, v_1^A, \dots ,v_{n-1}^A)$. We say that $\mathbb{G}$ has {\em $A$-height $n$} if $\mathbb{G}$ has $A$-height $\geq n$ but not $A$-height $\geq n+1$. If $\mathbb{G}$ has $A$-height $\geq n$ for all $n$, then we say that $\mathbb{G}$ has $A$-height $\infty$.
\item 
The inclusion $V^A\rightarrow L^A$ associates, to each formal $A$-module, an $A$-typical formal $A$-module $\typ(\mathbb{G})$ which is isomorphic to it. If $\mathbb{G}$ is an arbitrary (not necessarily $A$-typical) formal $A$-module, we say that $\mathbb{G}$ has {\em $A$-height $n$} if $\typ(\mathbb{G})$ has $A$-height $n$.
\end{itemize}
\end{definition}

The following is proven in~\cite{MR2987372}:
\begin{prop}\label{height of underlying fgl}
Let $p$ be a prime number.
\begin{itemize}
\item Every formal group law over a commutative $\hat{\mathbb{Z}}_p$-algebra admits the unique structure of a formal $\hat{\mathbb{Z}}_p$-module. Consequently, there is an equivalence of categories between formal group laws over commutative $\hat{\mathbb{Z}}_p$-algebras, and formal $\hat{\mathbb{Z}}_p$-modules. Under this correspondence, a formal $\hat{\mathbb{Z}}_p$-module is $\hat{\mathbb{Z}}_p$-typical if and only if its underlying formal group law is $p$-typical. If $\mathbb{G}$ is a formal $\hat{\mathbb{Z}}_p$-module of $\hat{\mathbb{Z}}_p$-height $n$, then its underlying formal group law has $p$-height $n$.
\item If $L,K$ are finite extensions of $\mathbb{Q}_p$ with rings of integers $B,A$ respectively, if $L/K$ is a field extension of degree $d$, and if $\mathbb{G}$ is a formal $B$-module of $B$-height $n$, then the underlying formal $A$-module of $\mathbb{G}$ has $A$-height $dn$.
\item In particular, if $K/\mathbb{Q}_p$ is a field extension of degree $d$, if $K$ has ring of integers $A$, and if $\mathbb{G}$ is a formal $A$-module of $A$-height $n$, then the underlying formal group law of $\mathbb{G}$ has $p$-height $dn$.
Consequently, {\em the only formal groups which admit complex multiplication by $A$ have underlying formal groups of $p$-height divisible by $d$.}
\end{itemize}
\end{prop}

Let $K/\mathbb{Q}_p$ be a finite field extension, let $A$ be the ring of integers of $K$, and let $k$ be the residue field of $A$.
It follows from the formal module generalization of Lazard's classification of formal group laws over $\overline{k}$ (see~\cite{MR0393050}) 
that, for each positive integer $n$, there exists
an $A$-height $n$ formal $A$-module over $\overline{k}$, and any two $A$-height $n$ formal $A$-modules over $\overline{k}$ are isomorphic.
Under the Barsotti-Tate module generalization of the Dieudonn\'{e}-Manin classification of $p$-divisible groups over $\overline{k}$ (see~\cite{MR0157972}; also see~\cite{MR1393439} for a nice treatment of the theory of Barsotti-Tate modules), it seems to be standard to write $\mathbb{G}_{r/s}$ for any indecomposable, slope $r/s$ Barsotti-Tate $A$-module over $\overline{k}$. Every formal $A$-module of $A$-height $n$ over $\overline{k}$, for $n$ positive and finite, has an associated Barsotti-Tate $A$-module, and this Barsotti-Tate $A$-module is indecomposable and isomorphic to $\mathbb{G}_{1/n}$ (if we were allowing formal $A$-modules of dimension greater than one, then we would get associated Barsotti-Tate modules which are not necessarily indecomposable, and with summands having slope of numerator greater than one).

To make computations of stable homotopy groups, we can work with formal modules defined over $\overline{\mathbb{F}}_p$, but typically this leads to the need to run an additional spectral sequence, namely a descent spectral sequence for the resulting $\Gal(\overline{\mathbb{F}}_p/\mathbb{F}_p)$-action on all the homotopy groups we compute. It is more convenient to work with formal modules over $\mathbb{F}_p$ whenever possible. To that end, we want to choose a $\Gal(\overline{k}/k)$-form of $\mathbb{G}_{1/n}$ for each $n$. In more concrete terms: we want to give a name to one particular convenient choice of $A$-height $n$ formal $A$-module over $k$ for each positive integer $n$. Here is our notation for this, in the particular cases we will study most closely in this document:
\begin{convention}
For each prime $p$ and each positive integer $n$, let $\mathbb{G}^{\hat{\mathbb{Z}}_p\left[ \sqrt{p}\right]}_{1/n}$ denote the $\hat{\mathbb{Z}}_p\left[ \sqrt{p}\right]$-typical formal $\hat{\mathbb{Z}}_p\left[ \sqrt{p}\right]$-module over $\mathbb{F}_p$
classified by the map $V^{\hat{\mathbb{Z}}_p\left[ \sqrt{p}\right]} \rightarrow \mathbb{F}_p$ sending $v_n^{\hat{\mathbb{Z}}_p\left[ \sqrt{p}\right]}$ to $1$ and sending $v_i^{\hat{\mathbb{Z}}_p\left[ \sqrt{p}\right]}$ to $0$ for all $i\neq n$.
We let $\mathbb{G}_{1/2n}$ denote the underlying formal group of $\mathbb{G}^{\hat{\mathbb{Z}}_p\left[ \sqrt{p}\right]}_{1/n}$.
\end{convention}
It is clear from Definition~\ref{def of typicality and height} that $\mathbb{G}^{\hat{\mathbb{Z}}_p\left[ \sqrt{p}\right]}_{1/n}$ has $\hat{\mathbb{Z}}_p\left[ \sqrt{p}\right]$-height $n$. In general, for a totally ramified finite extension $L/K$, the underlying formal $\mathcal{O}_K$-module of an $\mathcal{O}_L$-typical formal $\mathcal{O}_L$-module is also $\mathcal{O}_K$-typical (see~\cite{cmah4} for this fact). Consequently $\mathbb{G}_{1/2n}$ is $p$-typical, and furthermore has $p$-height $2n$, by Proposition~\ref{height of underlying fgl}. 
It follows from a computation of the map $BP_*\rightarrow V^{\hat{\mathbb{Z}}_p\left[ \sqrt{p}\right]}$ (e.g. in~\cite{cmah4} or in~\cite{formalmodules4}) that $\mathbb{G}_{1/2n}$ is the $p$-typical formal group law over $\mathbb{F}_p$ classified by the map $BP_*\rightarrow \mathbb{F}_p$ sending $v_n$ to $1$ and sending $v_i$ to $0$ for all $i\neq n$. 

\begin{theorem}\label{relative ravenel-may ss} 
There exist spectral sequences
\begin{align}
\label{relative may 1 ss} E_2^{s,t,u,v} \cong H^{s,t,u}\left(\mathcal{K}(2n,\floor{\frac{2np}{p-1}})/I_{n,m}\right) \otimes_{\mathbb{F}_p} P\left(b_{i,j}: 1\leq i \leq \floor{\frac{2n}{p-1}}, 0\leq j< n\right) &\Rightarrow 
\Cotor_{E_0\mathbb{F}_p[\strictAut(\mathbb{G}^{\hat{\mathbb{Z}}_p\left[ \sqrt{p}\right]}_{1/n})]^*}^{s,t,u}(\mathbb{F}_p,\mathbb{F}_p) \\
\nonumber d_r: E_r^{s,t,u,v} &\rightarrow E_r^{s+1,t,u,v+r-1} \\
\label{relative may 2 ss} \Cotor_{E_0\mathbb{F}_p[\strictAut(\mathbb{G}^{\hat{\mathbb{Z}}_p\left[ \sqrt{p}\right]}_{1/n})]^*}^{s,t,u}(\mathbb{F}_p,\mathbb{F}_p) &\Rightarrow H^s(\strictAut(\mathbb{G}^{\hat{\mathbb{Z}}_p\left[ \sqrt{p}\right]}_{1/n}); \mathbb{F}_p)\\
\nonumber d_r: E_r^{s,t,u} &\rightarrow E_r^{s+1,t-r,u}  \end{align}
and maps of spectral sequences
\[\xymatrix{
 H^{*,*,*}(\mathcal{K}(2n,\floor{\frac{2np}{p-1}})) \otimes_{\mathbb{F}_p} P(b_{i,j}: 1\leq i \leq \floor{\frac{2n}{p-1}}, 0\leq j< 2n) \ar@{=>}[r]\ar[d] & \Cotor_{E_0\mathbb{F}_p[\strictAut(\mathbb{G}_{1/2n})]^*}^{*,*,*}(\mathbb{F}_p,\mathbb{F}_p) \ar[d] \\
 H^{*,*,*}(\mathcal{K}(2n,\floor{\frac{2np}{p-1}})/I_{n,m}
  ) \otimes_{\mathbb{F}_p} P(b_{i,j}: 1\leq i \leq \floor{\frac{2n}{p-1}}, 0\leq j< n) \ar@{=>}[r] & \Cotor_{E_0\mathbb{F}_p[\strictAut(\mathbb{G}^{\hat{\mathbb{Z}}_p\left[ \sqrt{p}\right]}_{1/n})]^*}^{*,*,*}(\mathbb{F}_p,\mathbb{F}_p)  \\
 \Cotor_{E_0\mathbb{F}_p[\strictAut(\mathbb{G}_{1/2n})]^*}^{*,*,*}(\mathbb{F}_p,\mathbb{F}_p) \ar[d] \ar@{=>}[r] \ar[d] &
 H^*(\strictAut(\mathbb{G}_{1/2n}); \mathbb{F}_p) \ar[d] \\
 \Cotor_{E_0\mathbb{F}_p[\strictAut(\mathbb{G}^{\hat{\mathbb{Z}}_p\left[ \sqrt{p}\right]}_{1/n})]^*}^{*,*,*}(\mathbb{F}_p,\mathbb{F}_p) \ar@{=>}[r] &
 H^*(\strictAut(\mathbb{G}^{\hat{\mathbb{Z}}_p\left[ \sqrt{p}\right]}_{1/n}); \mathbb{F}_p) }\]
relating spectral sequences~\ref{relative may 1 ss} and~\ref{relative may 2 ss} to those of Theorem~\ref{ravenels may spectral sequences}.
\end{theorem}
\begin{proof}
It follows from the usual Hopf algebroid ``functor of points'' argument that
the automorphism group scheme $\strictAut(\mathbb{G}_{1/2n})$ is corepresented by
$\mathbb{F}_p[\strictAut(\mathbb{G}_{1/2n})]^* \cong \mathbb{F}_p\otimes_{BP_*}BP_*BP\otimes_{BP_*} \mathbb{F}_p$, with the action of $BP_*$ on $\mathbb{F}_p$ given by the ring map
$BP_*\rightarrow \mathbb{F}_p$ classifying $\mathbb{G}_{1/2n}$. 
In section~6.3 of~\cite{MR860042} one finds the computation
\[ \mathbb{F}_p\otimes_{BP_*}BP_*BP\otimes_{BP_*} \mathbb{F}_p \cong \mathbb{F}_p[t_1, t_2, \dots ]/\left(t_i^{p^{2n}} - t_i\mbox{\ \ for\ all\ } i\right) .\]
Now the automorphism group scheme $\strictAut(\mathbb{G}^{\hat{\mathbb{Z}}_p\left[ \sqrt{p}\right]}_{1/n})$ is corepresented by
$\mathbb{F}_p[\strictAut(\mathbb{G}^{\hat{\mathbb{Z}}_p\left[\sqrt{p}\right]}_{1/n})]^* \cong \mathbb{F}_p\otimes_{V^A}V^AT\otimes_{V^A} \mathbb{F}_p$,
and in~\cite{formalmodules4} and in~\cite{MR745362}, 
one finds the computation
$\mathbb{F}_p\otimes_{V^A}V^AT\otimes_{V^A} \mathbb{F}_p \cong \mathbb{F}_p\otimes_{BP_*}BP_*BP\otimes_{BP_*} \mathbb{F}_p/\left(t_i^{p^n} - t_i\mbox{\ \ for\ all\ } i\right)$.
Ravenel's filtration of $\mathbb{F}_p\otimes_{BP_*}BP_*BP\otimes_{BP_*} \mathbb{F}_p$ (see the proof of Theorem~\ref{ravenels may spectral sequences}) induces a filtration on the quotient algebra $\mathbb{F}_p\otimes_{V^A}V^AT\otimes_{V^A} \mathbb{F}_p$,
hence a map of resulting May spectral sequences; one checks that the associated graded algebra $E_0\left( \mathbb{F}_p[\strictAut(\mathbb{G}^{\hat{\mathbb{Z}}_p\left[\sqrt{p}\right]}_{1/n})]^* \right)$ is primitively generated and that Ravenel's splitting of the Lie-May spectral sequence~\ref{may 1 ss} is compatible with the quotient map $\mathbb{F}_p\otimes_{BP_*}BP_*BP\otimes_{BP_*}\mathbb{F}_p \rightarrow \mathbb{F}_p\otimes_{V^A}V^AT\otimes_{V^A} \mathbb{F}_p$, hence we have a map of resulting Lie-May spectral sequences. A more detailed proof of these claims, in greater generality (applying to a much larger class of extensions $K/\mathbb{Q}_p$ than those of the form $\mathbb{Q}_p(\sqrt{p})/\mathbb{Q}_p)$), appears in versions of~\cite{ravenelfilt} which were already privately circulated.
\end{proof}

\begin{definition}\label{def of e dga}
Let $p$ be a prime, 
$\ell$ a nonnegative integer, $m$ a positive integer with $\ell\leq m$, $n$ an even positive integer, 
and let $\mathcal{E}(n,m
,\ell)$ denote the bigraded 
$C_n$-equivariant differential graded subalgebra 
\begin{align*} \mathcal{E}(n,m
,\ell) &= \Lambda\left( h_{i,j}, w_{i^{\prime},j^{\prime}}: 1\leq i\leq m, 1\leq i^{\prime}\leq \ell, 0\leq j < n/2, 0\leq j^{\prime} \leq n/2 \right) \\
 & \subseteq E_0\mathcal{K}(n,m
)\end{align*}
of $E_0\mathcal{K}(n,m
)$.
\end{definition}

\begin{observation}\label{properties of e dga}
The important (and easy) properties of the differential graded algebras $\mathcal{E}(n,m,\alpha,\ell)$ are the following:
\begin{itemize}
\item $\mathcal{E}(n,m
,0) = \mathcal{K}(n,m)/I_{n,m}$ (as trigraded equivariant DGAs),
\item $\mathcal{E}(n,m
,m) = E_0\mathcal{K}(n,m
)$, (as trigraded equivariant DGAs),
\item $\mathcal{E}(2n,m,0) \cong \mathcal{K}(2n,m)/I_{2n,m} \cong \mathcal{K}(n,m)$
as {\em bigraded} equivariant DGAs (the right-hand isomorphism does not respect the Ravenel grading),
\item we have a short exact sequence of (equivariant, trigraded) differential graded algebras
\[ 1 \rightarrow \mathcal{E}(n,m
,\ell-1) \rightarrow \mathcal{E}(n,m
,\ell) \rightarrow \Lambda\left( w_{\ell,0},w_{\ell,1},\dots ,w_{\ell,n/2-1}\right) \rightarrow 1.\]
\end{itemize}
\end{observation}

Putting these ideas together, we have a strategy for computing the cohomology $H^*(\strictAut(\mathbb{G}_{1/2n}); \mathbb{F}_p)$ of the height $2n$ strict Morava stabilizer group scheme by first computing the cohomology 
$H^*(\strictAut(\mathbb{G}_{1/n}); \mathbb{F}_p)$ of the height $n$ strict Morava stabilizer group scheme:
\begin{strategy}\label{main strategy}
\begin{enumerate}
\item Compute $H^{*,*,*}(\mathcal{K}(n,\floor{\frac{np}{p-1}}))$.
(This cohomology is the input for Lie-May spectral sequence~\ref{may 1 ss}, which collapses with no differentials when $p>n+1$. The output of Lie-May spectral sequence~\ref{may 1 ss} is the input for May spectral sequence~\ref{may 2 ss}, which often collapses with no differentials for dimensional reasons. In particular, as per the remarks following Theorem~6.3.5 in~\cite{MR860042},
 when $n<p-1$ the smallest values of $n$ and $p$ for which any differentials are known to be nonzero in spectral sequence~\ref{may 2 ss} is $n=9$ and $p=11$. Consequently, when $n<p-1$, $H^{*,*,*}(\mathcal{K}(n,\floor{\frac{np}{p-1}}))$ is ``usually'' isomorphic to the cohomology $H^*(\strictAut(\mathbb{G}_{1/n}); \mathbb{F}_p)$ of the height $n$ strict Morava stabilizer group scheme.)
\item Using $H^{*,*,*}(\mathcal{K}(n,\floor{\frac{np}{p-1}}))$ as input, run the Cartan-Eilenberg spectral sequences for the extensions of trigraded equivariant DGAs
\[ 1 \rightarrow \mathcal{K}(n, m-1) \rightarrow \mathcal{K}(n, m) \rightarrow \Lambda(h_{m,0},h_{m,1}, \dots, h_{m,n-1}) \rightarrow 1\]
for $m = \floor{\frac{np}{p-1}}+1, \floor{\frac{np}{p-1}}+2, \dots , \floor{\frac{2np}{p-1}}$ to compute $H^{*,*,*}(\mathcal{K}(n,\floor{\frac{2np}{p-1}}))$.
(This cohomology is the input for Lie-May spectral sequence~\ref{relative may 1 ss}, which again collapses with no differentials when $p>n+1$. The output of Lie-May spectral sequence~\ref{relative may 1 ss} is the input for May spectral sequence~\ref{relative may 2 ss}, which again often collapses with no differentials for dimensional reasons, particularly when $n<p-1$. Consequently, when $n<p-1$, $H^{*,*,*}(\mathcal{K}(n,\floor{\frac{2np}{p-1}}))$ is ``usually'' isomorphic to the cohomology $H^*(\strictAut(\mathbb{G}^{\hat{\mathbb{Z}}_p\left[\sqrt{p}\right]}_{1/n}); \mathbb{F}_p)$ of the strict automorphism group scheme of a $\hat{\mathbb{Z}}_p\left[\sqrt{p}\right]$-height $n$ formal $\hat{\mathbb{Z}}_p\left[\sqrt{p}\right]$-module.)
\item Using $H^{*,*,*}(\mathcal{K}(n,\floor{\frac{2np}{p-1}})) \cong H^{*,*,*}(\mathcal{E}(2n,2\floor{\frac{np}{p-1}},0))$ as input (and adjusting the Ravenel grading to put every $h_{i,j}$ in Ravenel degree $d_{2n,i}$ instead of $d_{n,i}$; see Observation~\ref{properties of e dga}),
run the Cartan-Eilenberg spectral sequences for the extensions of trigraded equivariant DGAs
\[ 1 \rightarrow \mathcal{E}(2n,2\floor{\frac{np}{p-1}}
,\ell-1) \rightarrow \mathcal{E}(2n,2\floor{\frac{np}{p-1}}
,\ell) \rightarrow \Lambda\left( w_{\ell,0},w_{\ell,1},\dots ,w_{\ell,n/2-1}\right) \rightarrow 1.\]
for $\ell = 1, 2, \dots , 2\floor{\frac{np}{p-1}}$.
\item Using $H^{*,*,*,*}(\mathcal{E}(2n,2\floor{\frac{np}{p-1}},2\floor{\frac{np}{p-1}}) \cong E_0\mathcal{K}(2n,2\floor{\frac{np}{p-1}})$, run the spectral sequence
\begin{equation}\label{height-doubling ss} H^{*,*,*,*}(E_0\mathcal{K}(2n,2\floor{\frac{np}{p-1}})) \Rightarrow H^{*,*,*}(\mathcal{K}(2n,2\floor{\frac{np}{p-1}}))\end{equation}
 for the filtration of 
$\mathcal{K}(2n,2\floor{\frac{np}{p-1}})$ by powers of $I$.
\item Using $H^{*,*,*}(\mathcal{K}(2n,2\floor{\frac{np}{p-1}}))$, run
spectral sequence~\ref{may 1 ss} to compute $\Cotor_{E_0\mathbb{F}_p[\strictAut(\mathbb{G})]^*}^{*,*,*}(\mathbb{F}_p,\mathbb{F}_p)$.
\item Finally, using $\Cotor_{E_0\mathbb{F}_p[\strictAut(\mathbb{G})]^*}^{*,*,*}(\mathbb{F}_p,\mathbb{F}_p)$, run spectral sequence~\ref{may 2 ss} to compute
$H^*(\strictAut(\mathbb{G}_{1/2n}); \mathbb{F}_p)$.
\end{enumerate}
\end{strategy}
This strategy is actually usable to make explicit computations, and in the rest of this document I describe the case $n=2$ and $p>5$, i.e., I show how to 
use the cohomology of the height $2$ strict Morava stabilizer group scheme to compute the cohomology of the height $4$ strict Morava stabilizer group scheme, at primes $p>5$.

\section{Running the spectral sequences in the height $4$ case, $p>5$.}

\subsection{The cohomology of $\mathcal{K}(2,2)$ and the height $2$ Morava stabilizer group scheme $\mathbb{G}_{1/2}$.}

The material in this subsection is easy and well-known, appearing already in section~6.3 of~\cite{MR860042}. 
\begin{prop}\label{coh of K 2 2}
Suppose $p>2$.
Then we have an isomorphism of trigraded $C_2$-equivariant $\mathbb{F}_p$-algebras
\[ H^{*,*,*}(\mathcal{K}(2,2)) \cong \mathbb{F}_p\{ 1,h_{10},h_{11},h_{10}\eta_2,h_{11}\eta_2,h_{10}h_{11}\eta_2 \} \otimes_{\mathbb{F}_p} \Lambda(\zeta_2), \]
with tridegrees and the $C_2$-action as follows (remember that the internal degree is always reduced modulo $2(p^2-1)$):
\begin{equation}\label{degree chart 1}
\begin{array}{llllll}
\mbox{Coh.\ class}          & \mbox{Coh.\ degree} & \mbox{Int.\ degree} & \mbox{Rav.\ degree} & \mbox{Image\ under\ } \sigma \\
1                           & 0                   & 0                   & 0                   & 1  \\
h_{10}                      & 1                   & 2(p-1)              & 1                   & h_{11} \\
h_{11}                      & 1                   & 2p(p-1)             & 1                   & h_{10} \\
\zeta_2                    & 1                   & 0                   & 2                   & \zeta_2 \\
h_{10}\eta_2                & 2                   & 2(p-1)              & 3                   & -h_{11}\eta_2 \\
h_{11}\eta_2                & 2                   & 2p(p-1)             & 3                   & -h_{10}\eta_2 \\
h_{10}\zeta_2                & 2                   & 2(p-1)              & 3                   & h_{11}\zeta_2 \\
h_{11}\zeta_2                & 2                   & 2p(p-1)             & 3                   & h_{10}\zeta_2 \\
h_{10}h_{11}\eta_2          & 3                   & 0                   & 4                   & h_{10}h_{11}\eta_2 \\
h_{10}\eta_2\zeta_2         & 3                   & 2(p-1)              & 5                   & -h_{11}\eta_2\zeta_2 \\
h_{11}\eta_2\zeta_2         & 3                   & 2p(p-1)             & 5                   & -h_{10}\eta_2\zeta_2 \\
h_{10}h_{11}\eta_2\zeta_2    & 4                   & 0                   & 6                   & h_{10}h_{11}\eta_2\zeta_2 .
    \end{array} \end{equation}
where the cup products in $\mathbb{F}_p\{ 1,h_{10},h_{11},h_{10}\eta_2,h_{11}\eta_2,h_{10}h_{11}\eta_2 \}$ are all zero aside from the Poincar\'{e} duality cup products, i.e.,
each class has the obvious dual class such that the cup product of the two is
$h_{10}h_{11}\eta_2$, and the remaining cup products are all zero.
\end{prop}
\begin{proof}
We compute the Cartan-Eilenberg spectral sequence for the extension of $C_2$-equivariant trigraded DGAs
\[ 1 \rightarrow \mathcal{K}(2,1) \rightarrow
  \mathcal{K}(2,2)\rightarrow
  \Lambda(h_{20},h_{21}) \rightarrow 1.\]
Since the differential on $\mathcal{K}(2,1)$ is zero (see Definition~\ref{def of K dga}), $H^{*,*,*}(\mathcal{K}(2,1)) \cong \mathcal{K}(2,1) \cong \Lambda(h_{10},h_{11})$.
A change of $\mathbb{F}_p$-linear basis is convenient here: we will write $\zeta_2$ for the element $h_{20} + h_{21}\in \Lambda(h_{20},h_{21})$. (This notation for this particular element is standard. As far as I know, it began with M. Hopkins' work on the ``chromatic splitting conjecture,'' in which $\zeta_2$ plays a special role.)
We will write $\eta_2$ for the element $h_{20} - h_{21}$.

We have the differentials
\begin{eqnarray*} d\zeta_2 & = & 0, \\
 d\eta_2 & = & -2h_{10}h_{11}\end{eqnarray*} 

A convenient way to draw this spectral sequence is as follows:
\[\xymatrix{ 1 & \eta_2\ar[dddl] & \zeta_2 & \eta_2\zeta_2\ar[dddl] \\
 h_{10} & & & \\
 h_{11} & & & \\
 h_{10}h_{11} & & & ,}\]
when $p>2$, where arrows represent nonzero differentials. 
\end{proof}
The table~\ref{degree chart 1} has one row for each element in an $\mathbb{F}_p$-linear basis for the cohomology ring $H^{*,*,*}(\mathcal{K}(2,2))$, but from now on in this document, for the sake of brevity, when writing out similar tables for grading degrees of elements in the cohomology of a multigraded equivariant DGA, 
I will just give one row for each element in a set of generators for the cohomology ring of the DGA.

\begin{prop}\label{cohomology of ht 2 morava stab grp}
Suppose $p>3$. Then the cohomology $H^*(\strictAut(\mathbb{G}_{1/2}); \mathbb{F}_p)$ of the height $2$ strict Morava stabilizer group scheme is isomorphic, as a graded $\mathbb{F}_p$-vector space, to
\[ H^{*,*,*}(\mathcal{K}(2,2)) \cong \mathbb{F}_p\{ 1,h_{10},h_{11},h_{10}\eta_2,h_{11}\eta_2,h_{10}h_{11}\eta_2 \} \otimes_{\mathbb{F}_p} \Lambda(\zeta_2) \]
from Proposition~\ref{coh of K 2 2}.
The cohomological grading on $H^*(\strictAut(\mathbb{G}_{1/2}); \mathbb{F}_p)$ corresponds to the cohomological grading on $H^{*,*,*}(\mathcal{K}(2,2))$, so that $h_{10},h_{11},\zeta_2 \in H^1(\strictAut(\mathbb{G}_{1/2}); \mathbb{F}_p)$,
$h_{10}\eta_2,h_{11}\eta_2\in H^2(\strictAut(\mathbb{G}_{1/2}); \mathbb{F}_p)$, and so on.

The multiplication on $H^*(\strictAut(\mathbb{G}_{1/2}); \mathbb{F}_p)$ furthermore agrees with the multiplication on $H^{*,*,*}(\mathcal{K}(2,2))$, modulo the question of exotic multiplicative extensions, i.e., jumps in Ravenel filtration in the products of elements in $H^*(\strictAut(\mathbb{G}_{1/2}); \mathbb{F}_p)$. 
\end{prop}
\begin{proof}
Spectral sequence~\ref{may 1 ss} collapses immediately, since $p>3$ implies that
$1> \floor{\frac{2}{p-1}}$.
Hence $\Cotor_{E_0\mathbb{F}_p[\strictAut(\mathbb{G}_{1/2})]^*}^{*,*,*}(\mathbb{F}_p,\mathbb{F}_p) \cong H^{*,*,*}(\mathcal{K}(4,2))$.

We now run spectral sequence~\ref{may 2 ss}. This is, like all May spectral sequences, the spectral sequence of the filtration (in this case, Ravenel's filtration, described in the proof of Theorem~\ref{ravenels may spectral sequences}) on the cobar complex $C^{\bullet}(A)$ of a coalgebra $A$ induced by a filtration on the coalgebra itself. To compute differentials, we take an element $x\in H^*(C^{\bullet}(E_0A))$, lift it to a cochain $\overline{x}\in H^*(C^{\bullet}(A))$ whose image in the cohomology of the associated graded $H^*(E_0(C^{\bullet}(A))) \cong H^*(C^{\bullet}(E_0A))$ is $x$, and then evaluate the differential $d(\overline{x})$ in the cobar complex $C^{\bullet}(A)$. If $d(\overline{x}) = 0$, then $\overline{x}$ is a cocycle in the cobar complex $C^{\bullet}(A)$ and not merely in its associated graded $E_0C^{\bullet}(A)$, hence $\overline{x}$ represents a cohomology class in $H^*(C^{\bullet}(A))$; if $d(\overline{x}) \neq 0$, then we add correcting coboundaries of lower or higher (depending on whether the filtration is increasing or decreasing) filtration until we arrive at a cocycle which we recognize as a cohomology class in the spectral sequence's $E_1$-page.

It will be convenient to use the presentation 
\[ \mathbb{F}_p[t_{i,j}: i\geq 1, 0\leq j\leq 1]/\left( t_{i,j}^p \mbox{\ \ for\ all\ } i,j\right) \]
for 
$E_0\left(\mathbb{F}_p[\strictAut(\mathbb{G}_{1/2})]^*\right)  \cong E_0\left( \mathbb{F}_p[t_1, t_2, \dots ]/\left( t_i^{p^2} - t_i\mbox{\ \ for\ all\ } i\right)\right)$, where $t_{i,j}$ is the image in the associated graded of $t_i^{p^j}$.
The coproduct on $\mathbb{F}_p[t_{i,j}: i\geq 1, 0\leq j\leq 1]/\left( t_{i,j}^p \mbox{\ \ for\ all\ } i,j\right)$, inherited from that of $\mathbb{F}_p[\strictAut(\mathbb{G}_{1/2})]^*$, is given by
\[ \Delta(t_{i,j}) = \sum_{k=0}^i t_{k,j} \otimes t_{i-k, k+j}\]
for all $i < \floor{\frac{2p}{p-1}}$; see Theorem~6.3.2 of~\cite{MR860042} for this formula.
\begin{description}
\item[$h_{10},h_{11}$] The class $h_{10}$ is represented by $t_{1,0}$ in the cobar complex $C^{\bullet}\left(E_0\left(\mathbb{F}_p[\strictAut(\mathbb{G}_{1/2})]^*\right)\right)$, which lifts to $t_1$ in the cobar complex 
$C^{\bullet}\left(\mathbb{F}_p[\strictAut(\mathbb{G}_{1/2})]^*\right)$.
Since $t_1$ is a coalgebra primitive, i.e., a cobar complex $1$-cocycle, 
all May differentials are zero on $h_{1,0}$.
The $C_2$-equivariance of the spectral sequence then tells us that all May 
differentials also vanish on $h_{1,0}$.
\item[$\zeta_2$] There is no nonzero class in cohomological degree $2$ and internal degree $0$ for $\zeta_2$ to hit by a May differential of any length.
\item[$h_{10}\eta_2, h_{11}\eta_2$] The cohomology class $h_{10}\eta_2$ in the Koszul complex of the Lie algebra of primitives in $E_0\left(\mathbb{F}_p[\strictAut(\mathbb{G}_{1/2})]^*\right)$ (of which $\mathcal{K}(4,2)$ is a subcomplex) is represented by the $2$-cocycle $t_{1,0}\otimes t_{2,0} - t_{1,0}\otimes t_{2,1} - t_{1,0}\otimes t_{1,0}t_{1,1}$ in the cobar complex of $E_0\left(\mathbb{F}_p[\strictAut(\mathbb{G}_{1/2})]^*\right)$. This $2$-cocycle lifts to the $2$-cocycle
$t_{1}\otimes t_{2} - t_{1}\otimes t_{2}^p - t_{1}\otimes t_{1}^{p+1}$ in the cobar complex of $\mathbb{F}_p[\strictAut(\mathbb{G}_{1/2})]^*$.
Hence all May differentials vanish on $h_{10}\eta_2$, and by $C_2$-equivariance, also $h_{11}\eta_2$.
\end{description}
So the May differentials of all lengths vanish on the generators of the ring
$\Cotor_{E_0\mathbb{F}_p[\strictAut(\mathbb{G}_{1/2})]^*}^{*,*,*}(\mathbb{F}_p,\mathbb{F}_p)$. So 
$H^*(\strictAut(\mathbb{G}_{1/2}); \mathbb{F}_p) \cong \Cotor_{E_0\mathbb{F}_p[\strictAut(\mathbb{G}_{1/2})]^*}^{*,*,*}(\mathbb{F}_p,\mathbb{F}_p) \cong H^{*,*,*}(\mathcal{K}(2,2))$ as a graded $\mathbb{F}_p$-vector space. 
\end{proof}

\subsection{The cohomology of $\mathcal{E}(4,3,0),\mathcal{E}(4,4,0)$, and the automorphism group scheme of a $\hat{\mathbb{Z}}_p[\sqrt{p}]$-height $2$ formal $\hat{\mathbb{Z}}_p[\sqrt{p}]$-module.}

\begin{prop}\label{coh of E 4 3 0}
Suppose $p>3$.
Then we have an isomorphism of trigraded $C_2$-equivariant $\mathbb{F}_p$-algebras
\[ H^{*,*,*}(\mathcal{E}(4,3,0)) \cong \mathcal{A}_{4,3,0}\otimes_{\mathbb{F}_p} \Lambda(\zeta_2), \]
where 
\begin{dmath*} \mathcal{A}_{4,3,0} \cong 
\mathbb{F}_p\{ 1,h_{10},h_{11},h_{10}h_{30}, h_{11}h_{31}, e_{40}, \eta_2e_{40}, h_{10}\eta_2h_{30}, h_{11}\eta_2h_{31}, h_{10}\eta_2h_{30}h_{31}, h_{11}\eta_2h_{30}h_{31}, h_{10}h_{11}\eta_2h_{30}h_{31}\} ,\end{dmath*}
with tridegrees and the $C_2$-action as follows: 
\begin{equation}\label{degree chart 3}
\begin{array}{llllll}
\mbox{Coh.\ class}          & \mbox{Coh.\ degree} & \mbox{Int.\ degree} & \mbox{Rav.\ degree} & \mbox{Image\ under\ } \sigma \\
\hline \\
1                           & 0                   & 0                   & 0                   & 1  \\
h_{10}                      & 1                   & 2(p-1)              & 1                   & h_{11} \\
h_{11}                      & 1                   & 2p(p-1)             & 1                   & h_{10} \\
h_{10}h_{30}                 & 2                   & 4(p-1)              & 1+p                   & h_{11}h_{31} \\
h_{11}h_{31}                 & 2                   & 4p(p-1)             & 1+p                   & h_{10}h_{30} \\
e_{40}                      & 2                   & 0                   & 1+p                   & -e_{40}\\
\eta_2e_{40}                & 3                   & 0                   & 3+p                   & \eta_2e_{40}\\
h_{10}\eta_2h_{30}           & 3                   & 4(p-1)              & 3+p                   & -h_{11}\eta_2h_{31} \\
h_{11}\eta_2h_{31}           & 3                   & 4p(p-1)             & 3+p                   & -h_{10}\eta_2h_{30} \\
h_{10}\eta_2h_{30}h_{31}     & 4                   & 2(p-1)              & 3+2p                  & h_{11}\eta_2h_{30}h_{31} \\
h_{11}\eta_2h_{30}h_{31}     & 4                   & 2p(p-1)             & 3+2p                   & h_{10}\eta_2h_{30}h_{31} \\
h_{10}h_{11}\eta_2h_{30}h_{31} & 5                   & 0                 & 4+2p                   & -h_{10}h_{11}\eta_2h_{30}h_{31} \\
\hline \\
\zeta_2                    & 1                   & 0                   & 2                   & \zeta_2 ,
    \end{array} \end{equation}
where the cup products in $\mathcal{A}_{4,3,0}$ are all zero aside from the Poincar\'{e} duality cup products, i.e.,
each class has the dual class such that the cup product of the two is
$h_{10}h_{11}\eta_2h_{30}h_{31}$, and the remaining cup products are all zero. The classes in table~\ref{degree chart 3} are listed in order so that the class which is $n$ lines below $1$ is, up to multiplication by a unit in $\mathbb{F}_p$, the Poincar\'{e} dual of the class which is $n$ lines above $h_{10}h_{11}\eta_2h_{30}h_{31}$.
\end{prop}
\begin{proof}
We use the isomorphism $\mathcal{K}(2,2) \cong \mathcal{E}(4,2,0)$, from Observation~\ref{properties of e dga}, to
compute the Cartan-Eilenberg spectral sequence for the extension of $C_2$-equivariant trigraded DGAs
\[ 1 \rightarrow \mathcal{E}(4,2,0) \rightarrow
  \mathcal{E}(4,3,0)\rightarrow
  \Lambda(h_{30},h_{31}) \rightarrow 1.\]
We have the differentials
\begin{eqnarray*} dh_{30} & = & -h_{10}\eta_2, \\
 dh_{31} & = & h_{11}\eta_2,\\
 d(h_{30}h_{31}) & = & -h_{10}\eta_2h_{31} - h_{11}\eta_2h_{30}.
\end{eqnarray*}
and their products with classes in $H^{*,*,*}(\mathcal{E}(4,2,0))$.
The nonzero products are
\begin{eqnarray*} 
 d(h_{11}h_{30}) & = & -h_{10}h_{11}\eta_2 \\
 d(h_{10}h_{31}) & = & -h_{10}h_{11}\eta_2 \\
 d(h_{10}h_{30}h_{31}) & = & h_{10}h_{11}\eta_2h_{30} \\
 d(h_{11}h_{30}h_{31}) & = & -h_{10}h_{11}\eta_2h_{31}. 
\end{eqnarray*}
We write $e_{40}$ for the cocycle $h_{10}h_{31} - h_{11}h_{30}$.
Extracting the output of the spectral sequence from knowledge of the differentials is routine. The Ravenel degrees in table~\ref{degree chart 3}, in principle, could have differed from those in table~\ref{degree chart 1} because we have reindexed the cohomology classes to use the Ravenel numbers $d_{4,i}$ instead of $d_{2,i}$; see the remark in Observation~\ref{properties of e dga}. However, since the only Koszul generators $h_{i,j}$ appearing in table~\ref{degree chart 1} have $i\leq 2$ and since $d_{2,i} = d_{4,i}$ for $i\leq 2$, the Ravenel degrees for the classes that are in both table~\ref{degree chart 1} and table~\ref{degree chart 3} are equal.
\end{proof}

\begin{prop}\label{coh of E 4 4 0}
Suppose $p>3$.
Then we have an isomorphism of trigraded $C_2$-equivariant $\mathbb{F}_p$-algebras
\[ H^{*,*,*}(\mathcal{E}(4,4,0)) \cong \mathcal{A}_{4,4,0} \otimes_{\mathbb{F}_p} \Lambda(\zeta_2,\zeta_4), \]
where 
\begin{dmath*}
\mathcal{A}_{4,4,0} \cong \mathbb{F}_p\{ 1,h_{10},h_{11},h_{10}h_{30}, h_{11}h_{31}, h_{10}\eta_4-\eta_2h_{30}, h_{11}\eta_4-\eta_2h_{31}, \eta_2e_{40}, h_{10}\eta_2h_{30}, h_{11}\eta_2h_{31}, h_{10}h_{30}\eta_4, h_{11}h_{31}\eta_4, \eta_4e_{40}+4\eta_2h_{30}h_{31}, h_{10}\eta_2h_{30}h_{31}, h_{11}\eta_2h_{30}h_{31}, h_{10}\eta_2h_{30}\eta_4, h_{11}\eta_2h_{31}\eta_4, h_{10}\eta_2h_{30}h_{31}\eta_4, h_{11}\eta_2h_{30}h_{31}\eta_4, h_{10}h_{11}\eta_2h_{30}h_{31}\eta_4\} ,\end{dmath*}
with tridegrees and the $C_2$-action as follows: 
\begin{equation}\label{degree chart 4}
\begin{array}{llllll}
\mbox{Coh.\ class}          & \mbox{Coh.\ degree} & \mbox{Int.\ degree} & \mbox{Rav.\ degree} & \mbox{Image\ under\ } \sigma \\
\hline \\
1                                 & 0                   & 0                   & 0                   & 1  \\
h_{10}                            & 1                   & 2(p-1)              & 1                   & h_{11} \\
h_{11}                            & 1                   & 2p(p-1)             & 1                   & h_{10} \\
h_{10}h_{30}                       & 2                   & 4(p-1)              & 1+p                   & h_{11}h_{31} \\
h_{11}h_{31}                       & 2                   & 4p(p-1)             & 1+p                   & h_{10}h_{30} \\
h_{10}\eta_4-\eta_2h_{30}          & 2                   & 2(p-1)              & 1+2p                 & -h_{11}\eta_4+\eta_2h_{30} \\
h_{11}\eta_4-\eta_2h_{31}          & 2                   & 2p(p-1)             & 1+2p                 & -h_{10}\eta_4+\eta_2h_{31} \\
\eta_2e_{40}                      & 3                   & 0                   & 3+p                   & \eta_2e_{40}\\
h_{10}\eta_2h_{30}                 & 3                   & 4(p-1)              & 3+p                   & -h_{11}\eta_2h_{31} \\
h_{11}\eta_2h_{31}                 & 3                   & 4p(p-1)             & 3+p                   & -h_{10}\eta_2h_{30} \\
h_{10}h_{30}\eta_4                 & 3                   & 4(p-1)              & 1+3p                  & -h_{11}h_{31}\eta_4 \\
h_{11}h_{31}\eta_4                 & 3                   & 4p(p-1)             & 1+3p                  & -h_{10}h_{30}\eta_4 \\
\eta_4e_{40}+4\eta_2h_{30}h_{31}   & 3                   & 0                   & 1+3p                   & \eta_4e_{40}+4\eta_2h_{30}h_{31}\\
h_{10}\eta_2h_{30}h_{31}           & 4                   & 2(p-1)              & 3+2p                  & h_{11}\eta_2h_{30}h_{31} \\
h_{11}\eta_2h_{30}h_{31}           & 4                   & 2p(p-1)             & 3+2p                   & h_{10}\eta_2h_{30}h_{31} \\
h_{10}\eta_2h_{30}\eta_4           & 4                   & 4(p-1)              & 3+3p                  & h_{11}\eta_2h_{31}\eta_4 \\
h_{11}\eta_2h_{31}\eta_4           & 4                   & 4p(p-1)             & 3+3p                  & h_{10}\eta_2h_{30}\eta_4 \\
h_{10}\eta_2h_{30}h_{31}\eta_4      & 5                   & 2(p-1)              & 3+4p                 & -h_{11}\eta_2h_{30}h_{31}\eta_4 \\
h_{11}\eta_2h_{30}h_{31}\eta_4      & 5                   & 2p(p-1)             & 3+4p                 & -h_{10}\eta_2h_{30}h_{31}\eta_4 \\
h_{10}h_{11}\eta_2h_{30}h_{31}\eta_4 & 6                   & 0                  & 4+4p                 & h_{10}h_{11}\eta_2h_{30}h_{31}\eta_4 \\
\hline \\
\zeta_2                          & 1                   & 0                   & 2                   & \zeta_2 \\
\zeta_4                          & 1                   & 0                   & 2p                  & \zeta_4 .
    \end{array} \end{equation}
The classes in table~\ref{degree chart 4} are listed in order so that the class which is $n$ lines below $1$ is, up to multiplication by a unit in $\mathbb{F}_p$, the Poincar\'{e} dual of the class which is $n$ lines above $h_{10}h_{11}\eta_2h_{30}h_{31}\eta_4$.
\end{prop}
\begin{proof}
We use $H^{*,*,*}(\mathcal{E}(4,4,0))$, from Proposition~\ref{coh of E 4 3 0}, to
compute the Cartan-Eilenberg spectral sequence for the extension of $C_2$-equivariant trigraded DGAs
\[ 1 \rightarrow \mathcal{E}(4,3,0) \rightarrow
  \mathcal{E}(4,4,0)\rightarrow
  \Lambda(h_{40},h_{41}) \rightarrow 1.\]
A change of $\mathbb{F}_p$-linear basis is convenient here: we will write $\zeta_4$ for the element $h_{40} + h_{41}\in \Lambda(h_{40},h_{41})$, and we will write $\eta_4$ for $h_{40}-h_{41}$.
We have the differentials
\begin{eqnarray*} d\zeta_{4} & = & 0, \\
 d\eta_4 & = & h_{10}h_{31} + h_{30}h_{11} = e_{40},
\end{eqnarray*}
and one nonzero products with a class in $H^{*,*,*}(\mathcal{E}(4,3,0))$,
\begin{eqnarray*} 
 d(\eta_2e_{40}\eta_4) & = & h_{10}h_{11}\eta_2h_{30}h_{31}.
\end{eqnarray*}
Extracting the output of the spectral sequence from knowledge of the differentials is routine.
The three classes $h_{10}\eta_4, h_{11}\eta_4, \eta_4e_{40}$ in the $E_{\infty}$-term are not cocycles in $H^{*,*,*}(\mathcal{E}(4,4,0))$; adding
terms of lower Cartan-Eilenberg filtration to get cocycles yields the cohomology classes
$h_{10}\eta_4-\eta_2h_{30}, h_{11}\eta_4-\eta_2h_{31}, \eta_4e_{40}+4\eta_2h_{30}h_{31}$.
Note that this implies that there are nonzero multiplications in $\mathcal{A}_{4,4,0}$ other than those between each class and its Poincar\'{e} dual; for example, $h_{10}(h_{10}\eta_4 - \eta_{2}h_{30}) = -h_{10}\eta_2h_{30}$. 
\end{proof}

\begin{prop}\label{coh of ht 2 fm}
Suppose $p>5$. Then the cohomology $H^*(\strictAut(\mathbb{G}_{1/2}^{\hat{\mathbb{Z}}_p\left[\sqrt{p}\right]}); \mathbb{F}_p)$ of the
strict automorphism of the $\hat{\mathbb{Z}}_p\left[\sqrt{p}\right]$-height $2$ formal $\hat{\mathbb{Z}}_p\left[\sqrt{p}\right]$-module $\mathbb{G}^{\hat{\mathbb{Z}}_p\left[\sqrt{p}\right]}_{1/2}$ is isomorphic, as a graded $\mathbb{F}_p$-vector space, to
\[ H^{*,*,*}(\mathcal{E}(4,4,0)) \cong\mathcal{A}_{4,4,0} \otimes_{\mathbb{F}_p} \Lambda(\zeta_2,\zeta_4), \]
from Proposition~\ref{coh of E 4 4 0}.
The cohomological grading on $H^*(\strictAut(\mathbb{G}_{1/2}^{\hat{\mathbb{Z}}_p\left[\sqrt{p}\right]}); \mathbb{F}_p)$ corresponds to the cohomological grading on $H^{*,*,*}(\mathcal{E}(4,4,0))$.

The multiplication on $H^*(\strictAut(\mathbb{G}_{1/2}^{\hat{\mathbb{Z}}_p\left[\sqrt{p}\right]}); \mathbb{F}_p)$ furthermore agrees with the multiplication on $H^{*,*,*}(\mathcal{E}(4,4,0))$, modulo the question of exotic multiplicative extensions, i.e., jumps in Ravenel filtration in the products of elements in $H^*(\strictAut(\mathbb{G}_{1/2}^{\hat{\mathbb{Z}}_p\left[\sqrt{p}\right]}); \mathbb{F}_p)$. 

In particular, the Poincar\'{e} series expressing the $\mathbb{F}_p$-vector space dimensions of the grading degrees in $H^*(\strictAut(\mathbb{G}_{1/2}^{\hat{\mathbb{Z}}_p\left[\sqrt{p}\right]}); \mathbb{F}_p)$ is \[ (1+s)^2\left( 1 + 2s + 4s^2 + 6s^3 + 4s^4 + 2s^5 + s^6\right).\]
\end{prop}
\begin{proof}
We use Observation~\ref{properties of e dga} for the isomorphism $H^{*,*,*}(\mathcal{E}(4,4,0)) \cong H^{*,*,*}(\mathcal{K}(4,4)/I_{4,4})$, which appears in the input of spectral sequence~\ref{relative may 1 ss}.
Spectral sequence~\ref{relative may 1 ss} collapses immediately, since $p>5$ implies that
$1> \floor{\frac{4}{p-1}}$.
Hence $\Cotor_{E_0\mathbb{F}_p[\strictAut(\mathbb{G}^{\hat{\mathbb{Z}}_p\left[\sqrt{p}\right]}_{1/2})]^*}^{*,*,*}(\mathbb{F}_p,\mathbb{F}_p) \cong H^{*,*,*}(\mathcal{E}(4,4,0))$.

We now run spectral sequence~\ref{relative may 2 ss}. See the proof of Proposition~\ref{cohomology of ht 2 morava stab grp} for the general method we use.
It will be convenient to use the presentation 
\[ \mathbb{F}_p[t_{i,j}: i\geq 1, 0\leq j\leq 1]/\left( t_{i,j}^p \mbox{\ \ for\ all\ } i,j\right) \]
for 
$E_0\left(\mathbb{F}_p[\strictAut(\mathbb{G}^{\hat{\mathbb{Z}}_p\left[\sqrt{p}\right]}_{1/2})]^*\right)  \cong E_0\left( \mathbb{F}_p[t_1, t_2, \dots ]/\left( t_i^{p^2} - t_i\mbox{\ \ for\ all\ } i\right)\right)$, where $t_{i,j}$ is the image in the associated graded of $t_i^{p^j}$.
The coproduct on $\mathbb{F}_p[t_{i,j}: i\geq 1, 0\leq j\leq 1]/\left( t_{i,j}^p \mbox{\ \ for\ all\ } i,j\right)$, inherited from that of $\mathbb{F}_p[\strictAut(\mathbb{G}^{\hat{\mathbb{Z}}_p\left[\sqrt{p}\right]}_{1/2})]^*$, is given by
\[ \Delta(t_{i,j}) = \sum_{k=0}^i t_{k,j} \otimes t_{i-k, k+j}\]
for all $i < \floor{\frac{4p}{p-1}}$; reduce the $n=4$ case of 
Theorem~6.3.2 of~\cite{MR860042} 
modulo the ideal generated by $t_i^{p^2} - t_i$, for all $i$, to arrive at this formula.
\begin{description}
\item[$h_{10},h_{11}, \zeta_2$] There are no nonzero May differentials of any length on these classes, by the same computation as in the proof of Proposition~\ref{cohomology of ht 2 morava stab grp}.
\item[$h_{10}h_{30}, h_{11}h_{31}$] The class $h_{10}h_{30}$ is represented by 
the $2$-cocycle \[ t_{1,0}\otimes t_{3,0} - t_{1,0}\otimes t_{1,0}t_{2,0} - \frac{1}{2} t_{1,0}^2\otimes t_{2,0} + \frac{1}{2} t_{1,0}^2 \otimes t_{2,1} - \frac{1}{2} t_{1,0}^2 \otimes t_{1,0}t_{1,1} - \frac{1}{3} t_{1,0}^3\otimes t_{1,1}\]
in the cobar complex $C^{\bullet}\left(E_0\left(\mathbb{F}_p[\strictAut(\mathbb{G}^{\hat{\mathbb{Z}}_p\left[ \sqrt{p}\right]}_{1/2})]^*\right)\right)$, 
which lifts to the $2$-cochain
\[ t_{1}\otimes t_{3} - t_{1}\otimes t_{1}t_{2} - \frac{1}{2} t_{1}^2\otimes t_{2} + \frac{1}{2} t_{1}^2 \otimes t_{2}^p - \frac{1}{2} t_{1}^2 \otimes t_{1}^{p+1} - \frac{1}{3} t_{1,0}^3\otimes t_{1}^p\]
in the cobar complex $C^{\bullet}\left(\mathbb{F}_p[\strictAut(\mathbb{G}_{1/2})]^*\right)$.
Since this $2$-cochain is also a $2$-cocycle, all May differentials vanish on $h_{1,0}h_{3,0}$. The $C_2$-equivariance of the spectral sequence then tells us that all May 
differentials also vanish on $h_{1,1}h_{3,1}$.
\item[$h_{10}\eta_4, h_{11}\eta_4$] The only elements of internal degree $2(p-1)$ and cohomological degree $3$ are scalar multiples of $h_{10}\zeta_2\zeta_4$, but $h_{10}\zeta_2\zeta_4$ is of higher Ravenel degree than $h_{10}\eta_4$. Hence $h_{10}\eta_4$ cannot support a May differential of any length. By $C_2$-equivariance, the same is true of $h_{11}\eta_4$.
\item[$\eta_2e_{40}$] The only elements of internal degree $0$ and cohomological degree $4$ are $\mathbb{F}_p$-linear combinations of 
$\zeta_2\eta_2e_{40}, \zeta_4\eta_2e_{40}, \zeta_2\eta_4e_{40},$ and $\zeta_4\eta_4e_{40},$, but all four of these elements have higher Ravenel degree than $\eta_2e_{40}$, so again $\eta_2e_{40}$ cannot support a May differential of any length.
\item[$h_{10}\eta_2h_{30},h_{11}\eta_2h_{31},\eta_4e_{40}$] Similar degree considerations eliminate the possibility of nonzero May differentials on these classes.
\item[$\zeta_4$]
The class $\zeta_4$ is represented by 
the $1$-cocycle \[ t_{4,0} + t_{4,1} - t_{1,0}t_{3,1} - t_{1,1}t_{3,0} - \frac{1}{2}t_{2,0}^2 - \frac{1}{2}t_{2,1}^2 + t_{1,0}t_{1,1}t_{2,0} + t_{1,0}t_{1,1}t_{2,1}- \frac{1}{2}t_{1,0}^2t_{1,1}^2,\]
in the cobar complex $C^{\bullet}\left(E_0\left(\mathbb{F}_p[\strictAut(\mathbb{G}^{\hat{\mathbb{Z}}_p\left[ \sqrt{p}\right]}_{1/2})]^*\right)\right)$, 
which lifts to the $1$-cochain
\[ t_{4} + t_{4}^p - t_{1}t_{3}^p - t_{1}^pt_{3} - \frac{1}{2}t_{2}^2 - \frac{1}{2}t_{2}^{2p} + t_{1}^{p+1}t_{2} + t_{1}^{p+1}t_{2}^p- \frac{1}{2}t_{1}^{2p+2},\]
in the cobar complex $C^{\bullet}\left(\mathbb{F}_p[\strictAut(\mathbb{G}_{1/2})]^*\right)$.
Since this $1$-cochain is also a $1$-cocycle, all May differentials vanish on $\zeta_4$.
\end{description}
Now suppose that $q\geq 1$ is some integer and that we have already shown that $d_r$ vanishes on all classes, for all $r<q$. Then 
$d_r(\eta_2e_{40}\cdot \eta_4e_{40}) = 0$, i.e., $d_r$ vanishes on the duality class in the algebra $\mathcal{A}_{4,4,0}$.
For each element in that algebra, we have shown that $d_r$ vanishes on either that element, or on its Poincar\'{e} dual. Since $d_r$ also vanishes on the duality class, $d_r$ vanishes on all elements in that algebra.
Since $d_r$ also vanishes on $\zeta_2$ and $\zeta_4$, $d_r$ vanishes on all classes. By induction, the spectral sequence collapses with no nonzero differentials.
So 
$H^*(\strictAut(\mathbb{G}^{\hat{\mathbb{Z}}_p\left[\sqrt{p}\right]}_{1/2}); \mathbb{F}_p) \cong \Cotor_{E_0\mathbb{F}_p[\strictAut(\mathbb{G}^{\hat{\mathbb{Z}}_p\left[\sqrt{p}\right]}_{1/2})]^*}^{*,*,*}(\mathbb{F}_p,\mathbb{F}_p) \cong H^{*,*,*}(\mathcal{E}(4,4,0))$ as a graded $\mathbb{F}_p$-vector space. 
\end{proof}

\subsection{The cohomology of $\mathcal{E}(4,4,1)$ and $\mathcal{E}(4,4,2)$.}

\begin{prop}\label{coh of E 4 4 1}
Suppose $p>3$.
Then we have an isomorphism of quad-graded $C_2$-equivariant $\mathbb{F}_p$-algebras
\[ H^{*,*,*,*}(\mathcal{E}(4,4,1)) \cong H^{*,*,*}(\mathcal{E}(4,4,0))\otimes_{\mathbb{F}_p} \Lambda(w_{1,0},w_{1,1}) ,\]
with elements in $H^{s,t,u}(\mathcal{E}(4,4,0))$ in quad-degree $(s,t,u,0)$, with $w_{1,j}$ in quad-degree $(1,1,2(p-1)p^j,1)$,
and with the $C_2$-action given by $\sigma(w_{1,i}) = w_{1,i+1}$ and $w_{1,2} = -w_{1,0}$.
\end{prop}
\begin{proof}
By Proposition~\ref{presentation for associated graded}, the differentials
on the classes $w_{1,i}$ are zero, so 
the Cartan-Eilenberg spectral sequence for the extension
\[ 1 \rightarrow \mathcal{E}(4,4,0) \rightarrow \mathcal{E}(4,4,1) \rightarrow \Lambda(w_{1,0},w_{1,1}) \rightarrow 1\]
collapses with no differentials.
\end{proof}

In Proposition~\ref{coh of E 4 4 2}, we can give a relatively compact description of the cohomology of the DGA $\mathcal{E}(4,4,2)$ if we give a set of generators
for $H^{*,*,*,*}(\mathcal{E}(4,4,2))$ as a $\Lambda(w_{10},w_{11})$-module, rather than an $\mathbb{F}_p$-linear basis.
We use the following notation for the $\Lambda(w_{10},w_{11})$-modules which arise:
\begin{itemize}
\item we write $Px$ for the free $\Lambda(w_{10},w_{11})$-module on a generator $x$ (here ``$P$'' stands for ``projective''),
\item we write $M_0x$ for $Px/w_{11}Px$,
\item we write $M_1x$ for $Px/w_{10}Px$,
\item we write $Tx$ for $Px/(w_{10},w_{11})Px$ (here ``$T$'' stands for ``truncated''), and
\item given a sequence of symbols $N_1x_1, N_2x_2, \dots ,N_nx_n$ of the above form, we write $x_0\{ N_1x_1, N_2x_2, \dots ,N_nx_n\}$ for the direct sum
$N_1x_0x_1\oplus N_2x_0x_2 \oplus \dots \oplus N_nx_0x_n$.
\end{itemize}
\begin{prop}\label{coh of E 4 4 2}
Suppose $p>3$.
Then we have an isomorphism of quad-graded $C_2$-equivariant $\mathbb{F}_p$-algebras between $H^{*,*,*,*}(\mathcal{E}(4,4,2))$
and the algebra:
\begin{align*} 
 1\{ P1, M_0w_{11}w\eta_2, M_1w_{10}w\zeta_2, Tw_{10}w_{11}w\eta_2w\zeta_2\} \\
 \oplus\ \  h_{10}\{ M_01, Pw\eta_2, M_0w\zeta_2, Pw\eta_2w\zeta_2\} \\
 \oplus\ \  h_{11}\{ M_11, M_1w\eta_2, Pw\zeta_2, Pw\eta_2w\zeta_2\} \\
 \oplus\ \  h_{10}h_{30}\{ P1, Pw\eta_2, Pw\zeta_2, Pw\eta_2w\zeta_2\} \\
 \oplus\ \  h_{11}h_{31}\{ P1, Pw\eta_2, Pw\zeta_2, Pw\eta_2w\zeta_2\} \\
 \oplus\ \  (h_{10}\eta_4 - \eta_2h_{30})\{ P1, M_0w_{11}w\eta_2, M_1w_{10}w\zeta_2, Tw_{10}w_{11}w\eta_2w\zeta_2\} \\
 \oplus\ \  (h_{11}\eta_4 - \eta_2h_{31})\{ P1, M_0w_{11}w\eta_2, M_1w_{10}w\zeta_2, Tw_{10}w_{11}w\eta_2w\zeta_2\} \\
 \oplus\ \  \eta_2e_{40}\{ T1, M_1w\eta_2, M_0w\zeta_2, Pw\eta_2w\zeta_2\} \\
 \oplus\ \  h_{10}\eta_2h_{30}\{ M_01, Pw\eta_2, M_0w\zeta_2, Pw\eta_2w\zeta_2\} \\
 \oplus\ \  h_{11}\eta_2h_{31}\{ M_11, M_1w\eta_2, Pw\zeta_2, Pw\eta_2w\zeta_2\} \\
 \oplus\ \  h_{10}h_{30}\eta_4\{ P1, Pw\eta_2, M_1w_{10}w\zeta_2, M_1w_{10}w\eta_2w\zeta_2\} \\
 \oplus\ \  h_{11}h_{31}\eta_4\{ P1, M_0w_{11}w\eta_2, Pw\zeta_2, M_0w_{11}w\eta_2w\zeta_2\} \\
 \oplus\ \  (\eta_4e_{40} + 4\eta_2h_{30}h_{31})\{ P1, M_0w_{11}w\eta_2, M_1w_{10}w\zeta_2, Tw_{10}w_{11}w\eta_2w\zeta_2\} \\
 \oplus\ \  h_{10}\eta_2h_{30}h_{31}\{ T1, M_1w\eta_2, M_0w\zeta_2, Pw\eta_2w\zeta_2\} \\
 \oplus\ \  h_{11}\eta_2h_{30}h_{31}\{ T1, M_1w\eta_2, M_0w\zeta_2, Pw\eta_2w\zeta_2\} \\
 \oplus\ \  h_{10}\eta_2h_{30}\eta_4\{ P1, Pw\eta_2, Pw\zeta_2, Pw\eta_2w\zeta_2\} \\
 \oplus\ \  h_{11}\eta_2h_{31}\eta_4\{ P1, Pw\eta_2, Pw\zeta_2, Pw\eta_2w\zeta_2\} \\
 \oplus\ \  h_{10}\eta_2h_{30}h_{31}\eta_4\{ P1, Pw\eta_2, M_1w_{10}w\zeta_2, M_1w_{10}w\eta_2w\zeta_2\} \\
 \oplus\ \  h_{11}\eta_2h_{30}h_{31}\eta_4\{ P1, Pw\eta_2, M_0w_{11}w\zeta_2, M_0w_{11}w\eta_2w\zeta_2\} \\
 \oplus\ \  h_{10}h_{11}\eta_2h_{30}h_{31}\eta_4\{ T1, M_1w\eta_2, M_0w\zeta_2, Pw\eta_2w\zeta_2\} 
\end{align*}
\end{prop}
\begin{proof}
We run the Cartan-Eilenberg spectral sequence for the extension
\[ 1 \rightarrow \mathcal{E}(4,4,1) \rightarrow \mathcal{E}(4,4,2) \rightarrow \Lambda(w_{2,0},w_{2,1}) \rightarrow 1.\]
A change of $\mathbb{F}_p$-linear basis is convenient here: let $w\zeta_2 = w_{20}+w_{21}$ and let $w\eta_2 = w_{20}-w_{21}$.
By Proposition~\ref{presentation for associated graded}, we have $d(w\eta_2) = 2h_{10}w_{11}$ and $d(w\zeta_2) = -2h_{11}w_{10}$.
Computing products of these differentials with classes in $\mathcal{E}(4,4,1)$ requires that we know products of the classes in $\mathcal{A}_{4,4,0}$ 
with $h_{10}$ and $h_{11}$. The nonzero products which are nontrivial (i.e., not procuts with $1$, and not Poincar\'{e} duality products) are:
\begin{align}
 \nonumber    \left( h_{10}\eta_4 - \eta_2h_{30}\right) h_{10} &= -h_{10}\eta_2h_{30}, \\
 \nonumber    \left( h_{11}\eta_4 - \eta_2h_{31}\right) h_{11} &= -h_{11}\eta_2h_{31}, \\
 \label{eq 1} \left( h_{10}\eta_4 - \eta_2h_{30}\right) h_{11} &= \frac{-3}{4} \eta_2e_{40} - \delta\left( \frac{1}{2}\eta_2\eta_4  - \frac{1}{2}h_{30}h_{31}\right), \\
 \label{eq 2} \left( h_{11}\eta_4 - \eta_2h_{31}\right) h_{10} &= \frac{3}{4} \eta_2e_{40} - \delta\left( \frac{-1}{2}\eta_2\eta_4  - \frac{1}{2}h_{30}h_{31}\right), \\
 \label{eq 3} \left( h_{10}h_{30}\eta_4 \right) h_{11}         &= -h_{10}\eta_2h_{30}h_{31} + \delta(\frac{1}{2} \eta_2h_{30}\eta_4) \\
 \label{eq 4} \left( h_{11}h_{31}\eta_4 \right) h_{10}         &=  h_{11}\eta_2h_{30}h_{31} - \delta(\frac{1}{2} \eta_2h_{31}\eta_4) \\
 \label{eq 5} \left(\eta_4 e_{40} + 4\eta_2h_{30}h_{31}\right)h_{10} &= -6h_{10}\eta_2h_{30}h_{31} + \delta(\eta_2h_{30}h_{31}) \\
 \label{eq 6} \left(\eta_4 e_{40} + 4\eta_2h_{30}h_{31}\right)h_{11} &= -6h_{11}\eta_2h_{30}h_{31} + \delta(\eta_2h_{30}h_{31}) ,\end{align}
where the symbol $\delta$ in~\ref{eq 1}, \ref{eq 2}, \ref{eq 3}, \ref{eq 4}, \ref{eq 5}, and \ref{eq 6} is the differential in the DGA $\mathcal{E}(4,4,2)$, not the differential in the spectral sequence.

Consequently, we get that the spectral sequence differentials are:
\begin{eqnarray*} 
 d(w\eta_2) &=& 2h_{10}w_{11} \\
 d(w\zeta_2) &=& -2h_{11}w_{10} \\
 d\left( (h_{10}\eta_4 - \eta_2h_{30})w\eta_2\right) &=& -2h_{10}\eta_2h_{30}w_{11} \\
 d\left( (h_{10}\eta_4 - \eta_2h_{30})w\zeta_2\right) &=& \frac{3}{2} \eta_2e_{40} w_{10} \\
 d\left( (h_{11}\eta_4 - \eta_2h_{31})w\eta_2\right) &=&  \frac{3}{2} \eta_2e_{40} w_{11} \\
 d\left( (h_{11}\eta_4 - \eta_2h_{31})w\zeta_2\right) &=& 2h_{11}\eta_2h_{31}w_{10} \\
 d\left( (h_{10}h_{30}\eta_4) w\zeta_2\right) &=& 2h_{10}\eta_2h_{30}h_{31}w_{10} \\
 d\left( (h_{11}h_{31}\eta_4) w\eta_2\right) &=& 2h_{11}\eta_2h_{30}h_{31}w_{11} \\
 d\left( (\eta_4e_{40} + 4\eta_2h_{30}h_{31}) w\zeta_2\right) &=& 12h_{11}\eta_2h_{30}h_{31}w_{10} \\
 d\left( (\eta_4e_{40} + 4\eta_2h_{30}h_{31}) w\eta_2\right) &=& -12h_{10}\eta_2h_{30}h_{31}w_{11} \\
 d\left( (h_{10}\eta_2h_{30}h_{31}\eta_4) w\zeta_2 \right) &=& -2h_{10}h_{11}\eta_2h_{30}h_{31}\eta_4w_{10} \\
 d\left( (h_{11}\eta_2h_{30}h_{31}\eta_4) w\eta_2 \right) &=& -2h_{10}h_{11}\eta_2h_{30}h_{31}\eta_4w_{11} \\
 d\left( w\eta_2w\zeta_2 \right) &=& 2h_{10}w_{11}w\zeta_2 + 2h_{11}w_{10}w\eta_2 \\
 d\left( (h_{10}\eta_4 - \eta_2h_{30}) w\eta_2w\zeta_2 \right) &=& -2h_{10}\eta_2h_{30}w_{11}w\zeta_2 - \frac{3}{2} \eta_2e_{40} w_{10} w\eta_2 \\
 d\left( (h_{11}\eta_4 - \eta_2h_{31}) w\eta_2w\zeta_2 \right) &=& 2h_{11}\eta_2h_{31}w_{10}w\eta_2 + \frac{3}{2} \eta_2e_{40} w_{11} w\zeta_2 \\
 d\left( (h_{10}h_{30}\eta_4) w\eta_2 w\zeta_2\right) &=& -2h_{10}\eta_2h_{30}h_{31}w_{10}w\eta_2 \\
 d\left( (h_{11}h_{31}\eta_4) w\eta_2 w\zeta_2\right) &=& 2h_{11}\eta_2h_{30}h_{31}w_{11}w\zeta_2 \\
 d\left( (\eta_4e_{40} + 4\eta_2h_{30}h_{31}) w\eta_2 w\zeta_2\right) &=& -12h_{10}\eta_2h_{30}h_{31}w_{11}w\zeta_2 - 12h_{11}\eta_2h_{30}h_{31}w_{10}w\eta_2 \\
 d\left( (h_{10}\eta_2h_{30}h_{31}\eta_4) w\eta_2 w\zeta_2 \right) &=& 2h_{10}h_{11}\eta_2h_{30}h_{31}\eta_4w_{10}w\eta_2 \\
 d\left( (h_{11}\eta_2h_{30}h_{31}\eta_4) w\eta_2 w\zeta_2 \right) &=& -2h_{10}h_{11}\eta_2h_{30}h_{31}\eta_4w_{11}w\zeta_2 ,
\end{eqnarray*}
along with products of these differentials with $w_{10}$ and $w_{11}$.
Computing the output of the spectral sequence from knowledge of the differentials is routine (although tedious).
\end{proof}

\subsection{The remaining spectral sequences.}

The remaining spectral sequences are more difficult to typeset due to their size, and especially by the time one reaches the last spectral sequence~\ref{may 2 ss}, some computations (of cocycle representatives for certain cohomology classes) are prohibitively difficult to do by hand, although not difficult to do by computer. I do not know yet what the best way is to visually depict these computations. (The absence of these spectral sequences from this document is the main reason that, as written several times already, this document is an ``announcement,'' not an actual paper.)

The spectral sequences remaining to be added to this document are:
\begin{enumerate}
\item the Cartan-Eilenberg spectral sequence for the extension
\[ 1 \rightarrow \mathcal{E}(4,4,2) \rightarrow \mathcal{E}(4,4,3)\rightarrow \Lambda(w_{30},w_{31}) \rightarrow 1,\]
\item the Cartan-Eilenberg spectral sequence for the extension
\[ 1 \rightarrow \mathcal{E}(4,4,3) \rightarrow \mathcal{E}(4,4,4)\rightarrow \Lambda(w_{40},w_{41}) \rightarrow 1,\]
\item the ``height-doubling'' spectral sequence~\ref{height-doubling ss}:
\[ H^{*,*,*,*}(\mathcal{E}(4,4,4))\cong H^{*,*,*,*}(E_0\mathcal{K}(4,4)) \Rightarrow H^{*,*,*}(\mathcal{K}(4,4)),\]
\item the Lie-May spectral sequence~\ref{may 1 ss}, which is easily seen to collapse immediately, yielding the isomorphism
\[ H^{*,*,*}(\mathcal{K}(4,4)) \cong \Cotor_{E_0\mathbb{F}_p[\strictAut(\mathbb{G}_{1/4})]^*}^{*,*,*}(\mathbb{F}_p,\mathbb{F}_p),\]
\item and the May spectral sequence~\ref{may 2 ss},
\[\Cotor_{E_0\mathbb{F}_p[\strictAut(\mathbb{G}_{1/4})]^*}^{*,*,*}(\mathbb{F}_p,\mathbb{F}_p) \Rightarrow H^*(\strictAut(\mathbb{G}_{1/4}); \mathbb{F}_p),\]
which also collapses with no differentials (although arriving at this conclusion requires substantial computation!).
\end{enumerate}

\section{What ought to happen at heights greater than four?}

For each nonnegative integer $n$ and each prime $p$ such that $p>n+1$, 
the $\mathbb{F}_p$-vector space dimension of $H^*(\strictAut(\mathbb{G}_{1/n}); \mathbb{F}_p)$ is divisible by $2^n$; since 
$H^*(\strictAut(\mathbb{G}_{1/n}); \mathbb{F}_p) \otimes_{\mathbb{F}_p} \mathbb{F}_p[v_n^{\pm 1}]$ is the input for the first of a sequence of $n$ Bockstein spectral sequence computing the height $n$ layer in the chromatic spectral sequence $E_1$-term (or, dually, the input $H^*(\Aut(\mathbb{G}_{1/n}); E(\mathbb{G}_{1/n})_*)$ for the descent spectral sequence used to compute $\pi_*(L_{K(n)}S)$), and since each one of these Bockstein spectral sequence halves the $\mathbb{F}_p$-vector space dimension of ``most'' of the $v_n$-periodized copies of $H^*(\strictAut(\mathbb{G}_{1/n}); \mathbb{F}_p)$, the quotient 
$\left(\dim_{\mathbb{F}_p}\left(H^*(\strictAut(\mathbb{G}_{1/n}); \mathbb{F}_p)\right)\right)/2^n$ is an important number: it gives us an excellent estimate of the ``size'' of the $n$th Greek letter family in the stable homotopy groups of spheres, at large primes.
Here is a table of these numbers, again assuming that $p>n+1$:
\begin{equation}\label{ranks of coh}
\begin{array}{lll}
\mbox{n}          & \dim_{\mathbb{F}_p}\left(H^*(\strictAut(\mathbb{G}_{1/n}); \mathbb{F}_p)\right)\mbox{\ \ \ \ } & \mbox{quotient\ by\ \ } 2^n \\
0 & 1 & 1 \\
1 & 2 & 1 \\
2 & 12 & 3 \\ 
3 & 152 & 19 \\
4 & 3440 & 215 \end{array}\end{equation}
The sequence $1,1,3,19,215$ is the initial sequence of Taylor coefficients of the generating function 
\begin{equation}\label{func eq} A(t) = \sum_{n\geq 0} \left( \frac{t^n}{n!}\prod_{k=1}^nA(kt)\right) ,\end{equation}
that is, if
\begin{equation}\label{taylor series} A(t) = \sum_{n\geq 0} \frac{a_n}{n!}t^n\end{equation}
satisfies~\ref{func eq}, then $a_0,a_1,a_2,a_3,a_4$ are equal to $1,1,3,19,215$, respectively. (Thanks to the contributors and editors of the Online Encyclopedia of Integer Sequences, without which I would not have found this generating function!)

\begin{conjecture}\label{recursion conjecture}
Let $a_0, a_1, \dots$ be the Taylor coefficients in the power series~\ref{taylor series} satisfying the equation~\ref{func eq}. Then the mod $p$ cohomology of the height $n$ strict Morava stabilizer group scheme at primes $p>n+1$ has total $\mathbb{F}_p$-vector space dimension $2^na_n$. That is,
\[ \dim_{\mathbb{F}_p}H^*\left(\strictAut(\mathbb{G}_{1/n}); \mathbb{F}_p\right) = 2^na_n.\]

In particular, table~\ref{ranks of coh} continues:
\begin{equation*}
\begin{array}{lll}
\mbox{n}          & \dim_{\mathbb{F}_p}\left(H^*(\strictAut(\mathbb{G}_{1/n}); \mathbb{F}_p)\right)\mbox{\ \ \ \ } & \mbox{ quotient\ by\ \ } 2^n \\
0 & 1 & 1 \\
1 & 2 & 1 \\
2 & 12 & 3 \\ 
3 & 152 & 19 \\
4 & 3440 & 215 \\
5 & 128512 & 4016 \\
6 & 7621888 & 119092 \\ 
7 & 704410240 & 5503205 \\
8 & 100647546112 & 393154477 .
\end{array}
\end{equation*}
\end{conjecture}

\bibliography{/home/asalch/texmf/tex/salch}{}
\bibliographystyle{plain}
\end{document}